\documentclass[12pt]{article}
\usepackage[latin1]{inputenc}
\usepackage{amssymb}
\usepackage{hyperref}
\usepackage{amsmath}
\usepackage{latexsym}
\usepackage{times}
\usepackage{cite}
\usepackage{amssymb}
\usepackage{amsfonts}
\usepackage{amscd}
\usepackage{xcolor}
\usepackage{amstext,amsmath,amssymb,amsfonts}
\newtheorem{theorem}{Theorem}

\newtheorem{corollary}[theorem]{Corollary}

\newtheorem{lemma}[theorem]{Lemma}

\newtheorem{proof}[theorem]{Proof}

\newcommand{\beq}{\begin{eqnarray}}
\newcommand{\eeq}{\end{eqnarray}}
\newcommand{\beqs}{\begin{eqnarray*}}
	\newcommand{\eeqs}{\end{eqnarray*}}
\newcommand{\bpro}{\begin{pro}}
	\newcommand{\epro}{\end{pro}}
\newcommand{\blem}{\begin{lem}}
	\newcommand{\elem}{\end{lem}}
\newcommand{\bdfn}{\begin{dfn}}
	\newcommand{\edfn}{\end{dfn}}
\newcommand{\bcor}{\begin{cor}}
	\newcommand{\ecor}{\end{cor}}
\newcommand{\bthm}{\begin{thm}}
	\newcommand{\ethm}{\end{thm}}
\newcommand{\bex}{\begin{ex}}
	\newcommand{\eex}{\end{ex}}
\newcommand{\brmk}{\begin{rmk}}
	\newcommand{\ermk}{\end{rmk}}
\newcommand{\bpr}{\begin{pr}}
	\newcommand{\epr}{\end{pr}}
\newcommand{\benum}{\begin{enumerate}}
	\newcommand{\eenum}{\end{enumerate}}
\newcommand{\bitem}{\begin{itemize}}
	\newcommand{\eitem}{\end{itemize}}

\chardef\bslash=`\\
\numberwithin{equation}{section}
\numberwithin{table}{section}
\numberwithin{theorem}{section}

\begin{document}
	\begin{center}
	{\Large {$\mathcal{R}(p,q)$-trinomial probability distribution: properties and particular cases}}\\
	\vspace{0,5cm}
	Fridolin Melong\\
	\vspace{0.25cm}
	{\em  Institut f\"ur Mathematik, Universit\"at Z\"urich,\\
		Winterthurerstrasse 190, CH-8057 Z\"urich, Switzerland}\\
	{\em fridomelong@gmail.com}
	\vspace{0.25cm}
\end{center}
\today

\vspace{0.5 cm}
\begin{abstract}
In this paper, we investigate the trinomial probability distribution of the first and second kind from  the  $\mathcal{R}(p,q)$-quantum algebras. Moreover, we compute their $\mathcal{R}(p,q)$-factorial moments and derive the corresponding covariance. Particular cases of trinomial probability distribution are deduced from the formalism developed.
\end{abstract}
{\noindent
	MSC (2020)17B37, 60E05, 05A30
}
\tableofcontents

\section{Introduction}
Charalambos presented the $q-$ deformed Vandermonde and Cauchy formulae. Moreover, the $q-$ deformed univariate discrete probability distributions were investigated.Their properties and limiting distributions were derived \cite{CA1}.

Furthermore, the $q-$ deformed multinomial coefficients was defined and their recurrence relations were deduced. Also, the $q-$ deformed multinomial  and negative $q-$ deformed multinomial probability distributions of the first and second kind  were presented \cite{CA2}. 

The same author extended the multivariate $q-$ deformed Vandermonde and Cauchy formulae, and  the multivariate $q-$ Pol\'ya and inverse  $q-$ Pol\'ya were constructed \cite{CA3}. Moreover, the $q-$ factorial moments of the bivariate discrete distributions were investigated \cite{CA4}.

Let $p$ and $q$ be two positive real numbers such that $ 0<q<p<1.$ We consider a meromorphic function $\mathcal{R}$ defined on $\mathbb{C}\times\mathbb{C}$ by\cite{HB}:\begin{equation}\label{r10}
\mathcal{R}(u,v)= \sum_{s,t=-l}^{\infty}r_{st}u^sv^t,
\end{equation}
with an eventual isolated singularity at the zero, 
where $r_{st}$ are complex numbers, $l\in\mathbb{N}\cup\left\lbrace 0\right\rbrace,$ $\mathcal{R}(p^n,q^n)>0,  \forall n\in\mathbb{N},$ and $\mathcal{R}(1,1)=0$ by definition. We denote by $\mathbb{D}_{R}$ the bidisk \begin{eqnarray*}
	\mathbb{D}_{R}
	&=&\left\lbrace w=(w_1,w_2)\in\mathbb{C}^2: |w_j|<R_{j} \right\rbrace,
\end{eqnarray*}
where $R$ is the convergence radius of the series (\ref{r10}) defined by Hadamard formula as follows \cite{TN}:
\begin{eqnarray*}
	\lim\sup_{s+t \longrightarrow \infty} \sqrt[s+t]{|r_{st}|R^s_1\,R^t_2}=1.
\end{eqnarray*} 

We denote by 
${\mathcal O}(\mathbb{D}_R)$ the set of holomorphic functions defined
on $\mathbb{D}_R.$

The  $\mathcal{R}(p,q)-$ deformed numbers is defined by \cite{HB}: 
\begin{eqnarray}\label{Rpqn}
[n]_{\mathcal{R}(p,q)}:= \mathcal{R}(p^n,q^n),\qquad n\in\mathbb{N},
\end{eqnarray}
the $\mathcal{R}(p,q)-$ deformed factorials and binomial coefficients are given  as:
\begin{eqnarray}\label{Rpqf}
[n]!_{\mathcal{R}(p,q)}:= \left\{\begin{array}{lr} 1 \quad \mbox{for} \quad n=0 \quad \\
\mathcal{R}(p,q)\cdots \mathcal{R}(p^n,q^n) \quad \mbox{for} \quad n\geq 1, \quad \end{array} \right.
\end{eqnarray}
and
\begin{small}
\begin{eqnarray}\label{Rpqbc}
\left[\begin{array}{c} n  \\ m\end{array} \right]_{\mathcal{R}(p,q)}:=
\frac{[n]!_{\mathcal{R}(p,q)}}{[m]!_{\mathcal{R}(p,q)}[n-m]!_{\mathcal{R}(p,q)}},\quad (n, m)\in\mathbb{N}\cup\{0\},\quad n\geq m.
\end{eqnarray}
\end{small}

We consider the following linear operators   on ${\mathcal O}(\mathbb{D}_R)$  given by:
\begin{eqnarray}\label{operat}
&&\quad Q: \varphi \longmapsto Q\varphi(z) := \varphi(qz)
\nonumber \\
&&\quad P: \varphi \longmapsto P\varphi(z): = \varphi(pz)
,
\end{eqnarray}
leading to define  the $\mathcal{R}(p,q)-$ deformed derivative:
\begin{equation}{\label{deriva1}}
\partial_{{\mathcal R},p,q} := \partial_{p,q}\frac{p - q}{P-Q}{\mathcal R}(P, Q)
= \frac{p - q}{pP-qQ}{\mathcal R}(pP, qQ)\partial_{p,q},
\end{equation}
where $\partial_{p,q}$ is the $(p,q)-$ derivative:
\begin{eqnarray}
\partial_{p,q}:\varphi \longmapsto
\partial_{p,q}\varphi(z) := \frac{\varphi(pz) - \varphi(qz)}{z(p-q)}.
\end{eqnarray}

The quantum algebra associated with the $\mathcal{R}(p,q)-$ deformation, denoted by 
$\mathcal{A}_{\mathcal{R}(p,q)}$ is generated by the
set of operators $\{1, A, A^{\dagger}, N\}$ satisfying the following
commutation relations \cite{HB1}:
\begin{eqnarray}
&& \label{algN1}
\quad A A^\dag= [N+1]_{\mathcal{R}(p,q)},\quad\quad\quad A^\dag  A = [N]_{\mathcal{R}(p,q)}.
\cr&&\left[N,\; A\right] = - A, \qquad\qquad\quad \left[N,\;A^\dag\right] = A^\dag
\end{eqnarray}
with its realization on ${\mathcal O}(\mathbb{D}_R)$ given by:
\begin{eqnarray*}\label{algNa}
	A^{\dagger} := z,\qquad A:=\partial_{\mathcal{R}(p,q)}, \qquad N:= z\partial_z,
\end{eqnarray*}
where $\partial_z:=\frac{\partial}{\partial z}$ is the usual derivative on $\mathbb{C}.$

The $\mathcal{R}(p,q)-$ deformed numbers \eqref{Rpqn} can be rewritten as follows \cite{HMD}:
\begin{eqnarray}
[n]_{\mathcal{R}(p,q)}=\frac{\phi^{n}_1-\phi^{n}_2}{ \phi_1-\phi_2}, \quad \phi_1\neq \phi_2,
\end{eqnarray}
where $\big(\phi_i\big)_{i\in\{1,2\}}$ are functions  of the parameters deformations $p$ and $q.$

The following relations hold \cite{HMRC}:
\begin{eqnarray}\label{011}
[x]_{\mathcal{R}(p^{-1},q^{-1})} &=& (\phi_1\,\phi_2)^{1-x}\,[x]_{\mathcal{R}(p,q)},\\
\,[r]_
{\mathcal{R}(p^{-1},q^{-1})}!&=& (\phi_1\,\phi_2)^{- {r  \choose 2}}\,\,[r]_{\mathcal{R}(p,q)}!,\label{014}\\
\,{[x]_{r,\mathcal{R}(p^{-1},q^{-1})}}
&=& (\phi_1\,\phi_2)^{-xr + {r +1 \choose 2}}\,{[x]_{r,\mathcal{R}(p,q)}}\label{015}.
\end{eqnarray}
For $a,b\in\mathbb{N},$ the  $\mathcal{R}(p,q)$- deformed shifted factorial is defined by \cite{HMD}:
\begin{equation*}\label{a}
\big(u \oplus v\big)^n_{\mathcal{R}(p,q)}: = \displaystyle \prod_{i=1}^{n}\big(u\,\phi^{i-1}_1 + v\,\phi^{i-1}_2\big),\quad \mbox{with}\quad  \big(u \oplus v\big)^0_{\mathcal{R}(p,q)}: =1.
\end{equation*}
{Analogously, 
	\begin{eqnarray*}
		\big(u \ominus v\big)^n_{\mathcal{R}(p,q)}: = \displaystyle \prod_{i=1}^{n}\big(u\,\phi^{i-1}_1 - b\,\phi^{i-1}_2\big),\quad \mbox{with}\quad  \big(u \ominus v\big)^0_{\mathcal{R}(p,q)}: =1.
	\end{eqnarray*}
	The $\mathcal{R}(p,q)$- deformed factorial of $u$ of order $r$ is defined by\cite{HMRC}:
	\begin{eqnarray*}
	[u]_{r,\mathcal{R
		}(p,q)}=\prod^{r}_{i=1}[u-i+1]_{\mathcal{R
		}(p,q)},\quad r\in\mathbb{N},
	\end{eqnarray*}

Furthermore, the generalized Vandermonde, Cauchy formulae and univariate probability distributions induced from the $\mathcal{R}(p,q)-$ deformed quantum algebras are investigated in \cite{HMD}. 

Our aims is to determine the trinomial probability distribution of the first and second kind with properties  associated to the  $\mathcal{R}(p,q)-$ deformed quantum algebras \cite{HB1}. 

This paper is organized as follows: In section $2,$ we presente the $\mathcal{R}(p,q)$- trinomial probability distribution of the first and second kind. Also their negative counterparts. Besides, the properties namely ($\mathcal{R}(p,q)$-factorial moments and covariance) are investigated. Section $3$ is reserved to deduce the relevant particular cases corresponding to some quantum algbras known in the literature.
\section{Generalized trinomial probability distribution}
The trinomial probability distribution of the first and second kind from the $\mathcal{R}(p,q)$-deformed quantum algebras are investigated. Their factorial moments are also computed. Besides, the $\mathcal{R}(p,q)$- covariance is deduced.
\subsection{ $\mathcal{R}(p,q)$-trinomial  distribution of the first kind}
The probability function of the $\mathcal{R}(p,q)$-random vector $\underline{Y}=\big(Y_1,Y_2\big)$ of the $\mathcal{R}(p,q)$- trinomial probability distribution of the first kind, with parameters $n,$ $\underline{\alpha}=\big(\alpha_1,\alpha_2\big),$ $p,$ and, $q$ is given by:
\begin{small}
	\begin{eqnarray*}
	P\big(Y_1=y_1,Y_2=y_2\big)=\genfrac{[}{]}{0pt}{}{n}{y_1,y_2}_{\mathcal {R}(p,q)}\frac{\alpha^{y_1}_1\alpha^{y_2}_2\phi^{{n-y_1\choose 2}+{n-y_2\choose 2}}_1\phi^{{y_1\choose 2}+{y_2\choose 2}}_2}{(1 \oplus \alpha_1)^{n}_{\mathcal{R}(p,q)}(1 \oplus \alpha_2)^{n-y_1}_{\mathcal{R}(p,q)}},
	\end{eqnarray*}
\end{small}
where $y_j\in\{0,1,\cdots,n\}, y_1+y_2\leq n, s_j=\displaystyle\sum_{i=1}^{j}y_j, 0<\alpha_j<1$, and $j\in\{1,2\}.$
\begin{theorem}
	The $\mathcal{R}(p,q)$-factorial moments of the  $\mathcal{R}(p,q)$-trinomial probability distribution of the first kind, with parameters $n,$ $\underline{\alpha}=\big(\alpha_1,\alpha_2\big),$ $p,$ and, $q$ are determined by the following relations:
	\begin{small}
		\begin{eqnarray}\label{tfk1}
		E\big([Y_1]_{m_1,\mathcal{R}(p,q)}\big)=\frac{\alpha^{m_1}_1\,\phi^{m_1\choose 2}_2\,[n]_{m_1,\mathcal{R}(p,q)}}{(1 \oplus \alpha_1)^{m_1}_{\mathcal{R}(p,q)}},\,m_1\in\{0,1,\cdots,n\},
		\end{eqnarray}
		\begin{eqnarray}\label{tfk2}
		E\big([Y_2]_{m_2,\mathcal{R}(p,q)}|Y_1=y_1\big)=\frac{\alpha^{m_2}_2\phi^{m_2\choose 2}_2\,[n-y_1]_{m_2,\mathcal{R}(p,q)}}{(1 \oplus \alpha_2)^{m_2}_{\mathcal{R}(p,q)}},
		\end{eqnarray}
		with $m_2\in\{0,1,\cdots,n-y_1\}.$
		\begin{eqnarray}\label{tfk3}
		E\big([Y_2]_{m_2,\mathcal{R}(p,q)}\big)=\frac{\alpha^{m_2}_2\,\phi^{{m_2\choose 2}+m_2(n-m_2)}_1\phi^{m_2\choose 2}_2\,[n]_{m_2,\mathcal{R}(p,q)}}{(1 \oplus \alpha_2)^{m_2}_{\mathcal{R}(p,q)}(\phi^{n-m_2}_1 \oplus \alpha_1\phi^{n-m_2}_2)^{m_2}_{\mathcal{R}(p,q)}},
		\end{eqnarray}
		with $m_2\in\{0,1,\cdots,n\}$ and 
		\begin{equation}\label{tfk4}
		E\big([Y_1]_{m_1,\mathcal{R}(p,q)}[Y_2]_{m_2,\mathcal{R}(p,q)}\big)
		=\frac{\phi^{{m_2\choose 2}+m_2(n-m_2)}_1\,\phi^{{m_1\choose 2}+{m_2 \choose 2}}_2\,\alpha^{m_1}_1\alpha^{m_2}_2[n]_{m_1+m_2,\mathcal{R}(p,q)}}{(1 \oplus \alpha_1)^{m_1}_{\mathcal{R}(p,q)}(1 \oplus \alpha_2)^{m_2}_{\mathcal{R}(p,q)}(\phi^{n-m_2}_1 \oplus \alpha_1\phi^{n-m_2}_2)^{m_2}_{\mathcal{R}(p,q)}},
		\end{equation}
		where $m_1\in\{0,1,\cdots, n-m_2\}$ and $m_2\in\{0,1,\cdots, n\}.$
	\end{small}
\end{theorem}
\begin{proof}
	The $\mathcal{R}(p,q)$- random variable $Y_1$ follows the $\mathcal{R}(p,q)$- binomial probability distribution of the first kind as:
	\begin{eqnarray}
	P\big(Y_1=y_1\big)=\genfrac{[}{]}{0pt}{}{n}{y_1}_{\mathcal {R}(p,q)}\frac{\alpha^{y_1}_1\phi^{{n-y_1\choose 2}}_1\phi^{{y_1\choose 2}}_2}{(1 \oplus \alpha_1)^{n}_{\mathcal{R}(p,q)}},\,y_1\in\{0,1,\cdots,n\}.
	\end{eqnarray}
	From \cite{HMD}, the $\mathcal{R}(p,q)$- factorial moments of $Y_1$ are given by \eqref{tfk1}. Moreover, the conditional probability distribution of $Y_2,$ given that $Y_1=y_1,$ is the $\mathcal{R}(p,q)$- binomial probability distribution of the first kind, with mass function :
	\begin{eqnarray*}
		P\big(Y_2=y_2|Y_1=y_1\big)=\genfrac{[}{]}{0pt}{}{n-y_1}{y_2}_{\mathcal {R}(p,q)}\frac{\alpha^{y_2}_2\phi^{{n-y_1-y_2\choose 2}}_1\phi^{{y_2\choose 2}}_2}{(1 \oplus \alpha_2)^{n-y_1}_{\mathcal{R}(p,q)}},\,y_2\in\{0,1,\cdots,n-y_1\}.
	\end{eqnarray*}
	Using also \cite{HMD}, the conditional $\mathcal{R}(p,q)$-factorial moments of $Y_2,$ given that $Y_1=y_1,$ are furnished by \eqref{tfk2}. Besides, we determine the $\mathcal{R}(p,q)$- factorial moments of $Y_2$  according to the formula:
	\begin{eqnarray*}
		E\big([Y_2]_{m_2,\mathcal{R}(p,q)}\big)&=&E\bigg(E\big([Y_2]_{m_2,\mathcal{R}(p,q)}|Y_1\big)\bigg)\nonumber\\&=& \frac{\alpha^{m_2}_2\phi^{m_2\choose 2}_2}{(1 \oplus \alpha_2)^{m_2}_{\mathcal{R}(p,q)}}\,E\big([n-Y_1]_{m_2,\mathcal{R}(p,q)}\big).
	\end{eqnarray*}
	Obviously,
	\begin{eqnarray*}
		E\big([n-Y_1]_{m_2,\mathcal{R}(p,q)}\big)=\sum_{y_1=0}^{n-m_2}[n-y_1]_{m_2,\mathcal{R}(p,q)}\genfrac{[}{]}{0pt}{}{n}{y_1}_{\mathcal {R}(p,q)}\frac{\alpha^{y_1}_1\phi^{{n-y_1\choose 2}}_1\phi^{{y_1\choose 2}}_2}{(1 \oplus \alpha_1)^{n}_{\mathcal{R}(p,q)}}.
	\end{eqnarray*}
	From the relation
	\begin{eqnarray*}
		[n-y_1]_{m_2,\mathcal{R}(p,q)}\,\genfrac{[}{]}{0pt}{}{n}{y_1}_{\mathcal {R}(p,q)}=[n]_{m_2,\mathcal{R}(p,q)}\,\genfrac{[}{]}{0pt}{}{n-m_2}{y_1}_{\mathcal {R}(p,q)}
	\end{eqnarray*}
	and the $\mathcal{R}(p,q)$-binomial formula \cite{HMRC}:
	\begin{eqnarray*}
		\big(1 \oplus t\big)^n_{\mathcal{R}(p,q)}=\sum_{k=0}^{n}\genfrac{[}{]}{0pt}{}{n}{k}_{\mathcal {R}(p,q)}\,\phi^{{n-k\choose 2}}_1\,\phi^{{k\choose 2}}_2\,t^{k},\,t\in\mathbb{R},
	\end{eqnarray*}
	we have:
	\begin{eqnarray*}
		E\big([n-Y_1]_{m_2,\mathcal{R}(p,q)}\big)&=&\sum_{y_1=0}^{n-m_2}[n]_{m_2,\mathcal{R}(p,q)}\,\genfrac{[}{]}{0pt}{}{n-m_2}{y_1}_{\mathcal {R}(p,q)}\frac{\alpha^{y_1}_1\phi^{{n-y_1\choose 2}}_1\phi^{{y_1\choose 2}}_2}{(1 \oplus \alpha_1)^{n}_{\mathcal{R}(p,q)}}\nonumber\\
		&=&\frac{\phi^{{m_2\choose 2}+m_2(n-m_2)}_1[n]_{m_2,\mathcal{R}(p,q)}\,(1 \oplus \alpha_1)^{n-m_2}_{\mathcal{R}(p,q)}}{(1 \oplus \alpha_1)^{n}_{\mathcal{R}(p,q)}}\nonumber\\
		&=& \frac{\phi^{{m_2\choose 2}+m_2(n-m_2)}_1[n]_{m_2,\mathcal{R}(p,q)}}{(\phi^{n-m_2}_1 \oplus \alpha_1\,\phi^{n-m_2}_2)^{m_2}_{\mathcal{R}(p,q)}}.
	\end{eqnarray*}
	Thus,
	\begin{eqnarray*}
		E\big([Y_2]_{m_2,\mathcal{R}(p,q)}\big)= \frac{\alpha^{m_2}_2\phi^{m_2 \choose 2}_2}{(1 \oplus \alpha_2)^{m_2}_{\mathcal{R}(p,q)}}\,\frac{[n]_{m_2,\mathcal{R}(p,q)}}{(\phi^{n-m_2}_1 \oplus \alpha_1\,\phi^{n-m_2}_2)^{m_2}_{\mathcal{R}(p,q)}}.
	\end{eqnarray*}
	So, the joint $\mathcal{R}(p,q)$- factorial moments $E\big([Y_1]_{m_1,\mathcal{R}(p,q)}[Y_2]_{m_2,\mathcal{R}(p,q)}\big), m_2\in\{0,1,\cdots, n-m_1\}, m_1\in\{0,1,\cdots,n\}$ can be calculated by applying the relation:
	\begin{eqnarray*}
		E\big([Y_1]_{m_1,\mathcal{R}(p,q)}[Y_2]_{m_2,\mathcal{R}(p,q)}\big)&=&E\bigg(E\big([Y_1]_{m_1,\mathcal{R}(p,q)}[Y_2]_{m_2,\mathcal{R}(p,q)}\big)|Y_1\bigg)\nonumber\\
		&=& \frac{\alpha^{m_2}_2\phi^{m_2 \choose 2}_2}{(1 \oplus \alpha_2)^{m_2}_{\mathcal{R}(p,q)}}\,E\big([Y_1]_{m_1,\mathcal{R}(p,q)}[n-Y_1]_{m_2,\mathcal{R}(p,q)}\big).
	\end{eqnarray*}
	Since
	\begin{eqnarray*}
		E\big([Y_1]_{m_1,\mathcal{R}(p,q)}[n-Y_1]_{m_2,\mathcal{R}(p,q)}\big)&=& \sum_{y_1=m_1}^{n-m_2}[y_1]_{m_1,\mathcal{R}(p,q)}[n-y_1]_{m_2,\mathcal{R}(p,q)}\nonumber\\
		&\times& \genfrac{[}{]}{0pt}{}{n}{y_1}_{\mathcal {R}(p,q)}\frac{\alpha^{y_1}_1\phi^{{n-y_1\choose 2}}_1\phi^{{y_1\choose 2}}_2}{(1 \oplus \alpha_1)^{n}_{\mathcal{R}(p,q)}}.
	\end{eqnarray*}
	Using the relation
	\begin{eqnarray*}
		[y_1]_{m_1,\mathcal{R}(p,q)}\,[n-y_1]_{m_2,\mathcal{R}(p,q)} \genfrac{[}{]}{0pt}{}{n}{y_1}_{\mathcal {R}(p,q)}=[n]_{m_1+m_2,\mathcal{R}(p,q)}\, \genfrac{[}{]}{0pt}{}{n-m_1-m_2}{y_1-m_1}_{\mathcal {R}(p,q)}
	\end{eqnarray*}
	and the $\mathcal{R}(p,q)$-binomial formula \cite{HMRC}, we obtain:
	\begin{small}
		\begin{eqnarray*}
			E\big([Y_1]_{m_1,\mathcal{R}(p,q)}[n-Y_1]_{m_2,\mathcal{R}(p,q)}\big)&=&[n]_{m_1+m_2,\mathcal{R}(p,q)}\alpha^{m_1}_1\phi^{{m_1\choose 2}}_2\nonumber\\
			&\times& \sum_{y_1=m_1}^{n-m_2-m_1} \genfrac{[}{]}{0pt}{}{n-m_1-m_2}{y_1-m_1}_{\mathcal {R}(p,q)}\frac{\big(\alpha_1\phi^{m_1}_2\big)^{y_1-m_1}\phi^{{y_1-m_1\choose 2}}_2}{(1 \oplus \alpha_1)^{n}_{\mathcal{R}(p,q)}}\nonumber\\
			&=&\frac{[n]_{m_1+m_2,\mathcal{R}(p,q)}\alpha^{m_1}_1\phi^{{m_1\choose 2}}_2\,(\phi^{m_1}_1 \oplus \alpha_1\phi^{m_1}_2)^{n-m_1-m_2}_{\mathcal{R}(p,q)}}{(1 \oplus \alpha_1)^{n}_{\mathcal{R}(p,q)}}.
		\end{eqnarray*}
		Thus, 
		\begin{eqnarray*}
			E\big([Y_1]_{m_1,\mathcal{R}(p,q)}[Y_2]_{m_2,\mathcal{R}(p,q)}\big)
			&=& \frac{[n]_{m_1+m_2,\mathcal{R}(p,q)}\alpha^{m_1}_1\alpha^{m_2}_2\phi^{{m_1\choose 2}+{m_2 \choose 2}}_2(\phi^{m_1}_1 \oplus \alpha_1\phi^{m_1}_2)^{n-m_1-m_2}_{\mathcal{R}(p,q)}}{(1 \oplus \alpha_2)^{m_2}_{\mathcal{R}(p,q)}(1 \oplus \alpha_1)^{n}_{\mathcal{R}(p,q)}}\nonumber\\
			&=&\frac{[n]_{m_1+m_2,\mathcal{R}(p,q)}\alpha^{m_1}_1\alpha^{m_2}_2\phi^{{m_1\choose 2}+{m_2 \choose 2}}_2}{(1 \oplus \alpha_1)^{m_1}_{\mathcal{R}(p,q)}(1 \oplus \alpha_2)^{m_2}_{\mathcal{R}(p,q)}(\phi^{n-m_2}_1 \oplus \alpha_1\phi^{n-m_2}_2)^{m_2}_{\mathcal{R}(p,q)}}.
		\end{eqnarray*}
	\end{small}
\end{proof}
\begin{corollary}
	The $\mathcal{R}(p,q)$- covariance of the $\mathcal{R}(p,q)$- trinomial probability distribution of the first kind is presented by:
	\begin{eqnarray*}
		Cov\big([Y_1]_{\mathcal{R}(p,q)},[Y_2]_{\mathcal{R}(p,q)}\big)
		=\frac{\phi^{n-1}_1\alpha_1\alpha_2[n]_{\mathcal{R}(p,q)}\big([n-1]_{\mathcal{R}(p,q)}-[n]_{\mathcal{R}(p,q)}\big)}{(1 + \alpha_1)(1 + \alpha_2)(\phi^{n-1}_1 + \alpha_1\phi^{n-1}_2)}.
	\end{eqnarray*}
\end{corollary}
\begin{proof}
	By definition, the covariance of $[Y_1]_{\mathcal{R}(p,q)}$ and $[Y_2]_{\mathcal{R}(p,q)}$ is given by 
	\begin{small}
	\begin{eqnarray*}
	Cov\big([Y_1]_{\mathcal{R}(p,q)},[Y_2]_{\mathcal{R}(p,q)}\big)=E\big([Y_1]_{\mathcal{R}(p,q)}[Y_2]_{\mathcal{R}(p,q)}\big)-E\big([Y_1]_{\mathcal{R}(p,q)}\big)E\big([Y_2]_{\mathcal{R}(p,q)}\big).
	\end{eqnarray*}
	\end{small}
Taking $m_1=m_2=1, $ in the relations \eqref{tfk1}, \eqref{tfk2}, and \eqref{tfk3}, we get 
\begin{eqnarray*}
E\big([Y_1]_{\mathcal{R}(p,q)}\big)E\big([Y_2]_{\mathcal{R}(p,q)}\big)=\frac{\alpha_1\,[n]_{\mathcal{R}(p,q)}}{(1 + \alpha_1)}\frac{\alpha_2\,\phi^{n-1}_1\,[n]_{\mathcal{R}(p,q)}}{(1 + \alpha_2)(\phi^{n-1}_1 + \alpha_1\phi^{n-1}_2)}
\end{eqnarray*}
and 
\begin{eqnarray*}
E\big([Y_1]_{\mathcal{R}(p,q)}[Y_2]_{\mathcal{R}(p,q)}\big)
=\frac{\phi^{n-1}_1\,\alpha_1\alpha_2[n]_{2,\mathcal{R}(p,q)}}{(1 + \alpha_1)(1 + \alpha_2)(\phi^{n-1}_1 + \alpha_1\phi^{n-1}_2)}.
\end{eqnarray*}
After computation, the result follows. 
\end{proof}

\subsection{Negative $\mathcal{R}(p,q)$-trinomial  distribution of the first kind}
Let $U_n$ be the number of successes until the occurrence of the $n^{th}$ failure, in a sequence of independent Bernoulli trials, with probability of success at the $i^{th}$ trial given as:
\begin{eqnarray*}
	p_i=\frac{\alpha\,\phi^{i-1}_2}{\phi^{i-1}_1 + \alpha\,\phi^{i-1}_2},\,i\in\mathbb{N}\cup\{0\}.
\end{eqnarray*}
\begin{lemma}
	The $\mathcal{R}(p,q)$- random variable $U_n$ obeys the negative $\mathcal{R}(p,q)$- binomial probability distribution of the first kind :
	\begin{eqnarray}\label{ntfka}
	P\big(U=u\big)=\genfrac{[}{]}{0pt}{}{n+u-1}{u}_{\mathcal {R}(p,q)}\frac{\alpha^{u}\,\phi^{n-u\choose 2}_1\phi^{u\choose 2}_2}{(1 \oplus \alpha)^{n+u}_{\mathcal{R}(p,q)}},\,u\in\mathbb{N}\cup\{0\}.
	\end{eqnarray} 
	Moreover, the $\mathcal{R}(p^{-1},q^{-1})$- factorial moments are presented by the relation
	\begin{eqnarray}\label{ntfkb}
	E\big([U_n]_{m,\mathcal{R}(p^{-1},q^{-1})}\big)=[n+m-1]_{m,\mathcal{R}(p^{-1},q^{-1})}\,\alpha^{m},\,m\in\mathbb{N}\cup\{0\}.
	\end{eqnarray}
\end{lemma}

We denote by $W_j$ the number of successes of the $j^{th}$ kind until the occurrence of the
nth failure of the second kind, in a sequence of Bernoulli trials with chain-composite
failures. 
Then, the $\mathcal{R}(p,q)$-random vector $\underline{W}=\big(W_1,W_2\big)$ follows the negative $\mathcal{R}(p,q)$- trinomial probability distribution of the first kind, with parameters $n,$ $\underline{\alpha}=\big(\alpha_1,\alpha_2\big),$ $p,$ and, $q.$ Its mass function is presented by:
\begin{small}
	\begin{eqnarray}\label{rpqmfk}
	\qquad P\big(\underline{W}=\underline{w}\big)=\genfrac{[}{]}{0pt}{}{n+w_1+w_2-1}{w_1,w_2}_{\mathcal {R}(p,q)}\frac{\alpha^{w_1}_1\alpha^{w_2}_2\phi^{{n-w_1\choose 2}+{n-w_2\choose 2}}_1\phi^{{w_1\choose 2}+{w_2\choose 2}}_2}{(1 \oplus \alpha_1)^{n+w_1+w_2}_{\mathcal{R}(p,q)}(1 \oplus \alpha_2)^{n+w_2}_{\mathcal{R}(p,q)}},
	\end{eqnarray}
\end{small}
where $w_j\in\mathbb{N}\cup\{0\}, 0<\alpha_j<1$, and $j\in\{1,2\}.$
\begin{theorem}
	For $m_1\in\mathbb{N}\cup\{0\}$ and $m_2\in\mathbb{N}\cup\{0\},$ the $\mathcal{R}(p^{-1},q^{-1})$-factorial moments of the negative $\mathcal{R}(p,q)$-trinomial probability distribution of the first kind, with parameters $n,$ $\underline{\alpha}=\big(\alpha_1,\alpha_2\big),$ $p,$ and, $q$ are given as follows:
	\begin{small}
		\begin{eqnarray}\label{ntfk1}
		E\big([W_2]_{m_2,\mathcal{R}(p^{-1},q^{-1})}\big)=\alpha^{m_2}_2\,[n+m_2-1]_{m_2,\mathcal{R}(p,q)},
		\end{eqnarray}
		\begin{eqnarray}\label{ntfk2}
		 E\big([W_1]_{m_1,\mathcal{R}(p^{-1},q^{-1})}|W_2=w_2\big)=\alpha^{m_1}_1\,[n+w_2+m_1-1]_{m_1,\mathcal{R}(p,q)},
		\end{eqnarray}
		\begin{eqnarray}\label{ntfk3}
		E\bigg(\frac{[W_1]_{m_1,\mathcal{R}(p^{-1},q^{-1})}}{\big(\phi^{n+W_2}_1 \oplus \alpha_2\,\phi^{n+W_2}_2\big)^{m_1}_{\mathcal{R}(p,q)}}\bigg)=\alpha^{m_1}_1\,[n+m_1-1]_{m_1,\mathcal{R}(p,q)},
		\end{eqnarray}
		and 
		\begin{small}
			\begin{equation}\label{ntfk4}
			E\bigg(\frac{[W_1]_{m_1,\mathcal{R}(p^{-1},q^{-1})}\,[W_2]_{m_2,\mathcal{R}(p^{-1},q^{-1})}}{\big(\phi^{n+W_2}_1 \oplus \alpha_2\,\phi^{n+W_2}_2\big)^{m_1}_{\mathcal{R}(p,q)}}\bigg)
			=\frac{[n+m_1+m_2-1]_{m_1+m_2,\mathcal{R}(p,q)}}{\alpha^{-m_1}_1\,\alpha^{-m_2}_2},
			\end{equation}
		\end{small}
	\end{small}
\end{theorem}
\begin{proof}
	According to the relations \eqref{ntfka} and \eqref{ntfkb}, we derive the $\mathcal{R}(p^{-1},q^{-1})$- factorial moments of $W_2$ given by \eqref{ntfk1}. Moreover, the conditional distribution of the $\mathcal{R}(p,q)$-random variable $W_1,$ given that $W_2=w_2,$ is a negative $\mathcal{R}(p,q)$-trinomial probability distribution of the first kind, with mass function:
	\begin{small}
		\begin{eqnarray*}
			P\big(W_1=w_1|W_2=w_2\big)=\genfrac{[}{]}{0pt}{}{n+w_1+w_2-1}{w_1}_{\mathcal {R}(p,q)}\frac{\alpha^{w_1}_1\phi^{n-w_1\choose 2}_1\phi^{w_1\choose 2}_2}{(1 \oplus \alpha_1)^{n+w_1+w_2}_{\mathcal{R}(p,q)}},\,m_1\in\mathbb{N}\cup\{0\}.
		\end{eqnarray*}
	\end{small}
	Using, once the relations \eqref{ntfka} and \eqref{ntfkb}, the conditional $\mathcal{R}(p^{-1},q^{-1})$- factorial moments of $W_1,$ given that $W_2=w_2,$ are given by \eqref{ntfk2}. 
	
	Furthermore, the  expected value of the $\mathcal{R}(p^{-1},q^{-1})$-function of $\underline{W}=\big(W_1,W_2\big)$
	\begin{eqnarray*}
		\Gamma:=\frac{[W_1]_{m_1,\mathcal{R}(p^{-1},q^{-1})}}{\big(\phi^{n+W_2}_1 \oplus \alpha_2\,\phi^{n+W_2}_2\big)^{m_1}_{\mathcal{R}(p,q)}},\,\, m_1\in\mathbb{N}\cup\{0\}
	\end{eqnarray*}
	can be calculated  according to the relation
	\begin{eqnarray*}
		E\big(\Gamma\big)&=& E\big(E\big(\Gamma|W_2\big)\big)\nonumber\\
		&=& \alpha^{m_1}_1\,E\bigg(\frac{[n+W_2+m_1-1]_{m_1,\mathcal{R}(p,q)}}{\big(\phi^{n+W_2}_1 \oplus \alpha_2\,\phi^{n+W_2}_2\big)^{m_1}_{\mathcal{R}(p,q)}}\bigg).
	\end{eqnarray*}
	Since
	\begin{eqnarray*}
		E\bigg(\frac{[n+W_2+m_1-1]_{m_1,\mathcal{R}(p,q)}}{\big(\phi^{n+W_2}_1 \oplus \alpha_2\,\phi^{n+W_2}_2\big)^{m_1}_{\mathcal{R}(p,q)}}\bigg)&=&\sum_{w_2}^{\infty}[n+w_2+m_1-1]_{m_1,\mathcal{R}(p,q)}\nonumber\\&\times&\genfrac{[}{]}{0pt}{}{n+w_2-1}{w_2}_{\mathcal {R}(p,q)}\frac{\alpha^{w_2}_2\phi^{n-w_2\choose 2}_1\phi^{w_2\choose 2}_2}{(1 \oplus \alpha_2)^{n+m_1+w_2}_{\mathcal{R}(p,q)}}.
	\end{eqnarray*}
	From the relation
	\begin{small}
		$$[n+w_2+m_1-1]_{m_1,\mathcal{R}(p,q)} \genfrac{[}{]}{0pt}{}{n+w_2-1}{w_2}_{\mathcal {R}(p,q)}=[n+m_1-1]_{m_1,\mathcal{R}(p,q)} \genfrac{[}{]}{0pt}{}{n+m_1+w_2-1}{w_2}_{\mathcal {R}(p,q)}$$
	\end{small}
	and the negative $\mathcal{R}(p,q)$-binomial formula \cite{HMRC}, we have:
	\begin{small}
		\begin{eqnarray*}
			E\bigg(\frac{[n+W_2+m_1-1]_{m_1,\mathcal{R}(p,q)}}{\big(\phi^{n+W_2}_1 \oplus \alpha_2\,\phi^{n+W_2}_2\big)^{m_1}_{\mathcal{R}(p,q)}}\bigg)&=&[n+m_1-1]_{m_1,\mathcal{R}(p,q)}\sum_{w_2=0}^{\infty}\genfrac{[}{]}{0pt}{}{n+m_1+w_2-1}{w_2}_{\mathcal {R}(p,q)}\nonumber\\&\times&\frac{\alpha^{w_2}_2\phi^{n-w_2\choose 2}_1\phi^{w_2\choose 2}_2}{(1 \oplus \alpha_2)^{n+m_1+w_2}_{\mathcal{R}(p,q)}}\nonumber\\
			&=&[n+m_1-1]_{m_1,\mathcal{R}(p,q)} .
		\end{eqnarray*}
	\end{small}
	Thus, the relation \eqref{ntfk3} holds. Analogously, the  expected value of the $\mathcal{R}(p^{-1},q^{-1})$-function of $\underline{W}=\big(W_1,W_2\big)$
	\begin{eqnarray*}
		\Lambda:=\frac{[W_1]_{m_1,\mathcal{R}(p^{-1},q^{-1})}[W_2]_{m_2,\mathcal{R}(p^{-1},q^{-1})}}{\big(\phi^{n+W_2}_1 \oplus \alpha_2\,\phi^{n+W_2}_2\big)^{m_1}_{\mathcal{R}(p,q)}},\,m_1\in\mathbb{N}\cup\{0\},\, m_2\in\mathbb{N}\cup\{0\}
	\end{eqnarray*}
	may be computed  using the relation
	\begin{eqnarray*}
		E\big(\Lambda\big)&=& E\big(E\big(\Lambda|W_2\big)\big)\nonumber\\
		&=& \alpha^{m_1}_1\,E\bigg(\frac{[W_2]_{m_2,\mathcal{R}(p^{-1},q^{-1})}[n+W_2+m_1-1]_{m_1,\mathcal{R}(p,q)}}{\big(\phi^{n+W_2}_1 \oplus \alpha_2\,\phi^{n+W_2}_2\big)^{m_1}_{\mathcal{R}(p,q)}}\bigg).
	\end{eqnarray*}
	Since
	\begin{small}
		\begin{eqnarray*}
			E\bigg(\frac{[W_2]_{m_2,\mathcal{R}(p^{-1},q^{-1})}[n+W_2+m_1-1]_{m_1,\mathcal{R}(p,q)}}{\big(\phi^{n+W_2}_1 \oplus \alpha_2\,\phi^{n+W_2}_2\big)^{m_1}_{\mathcal{R}(p,q)}}\bigg)&=&\sum_{w_2=m_2}^{\infty}\frac{[w_2]_{m_2,\mathcal{R}(p^{-1},q^{-1})}[n+w_2+m_1-1]_{m_1,\mathcal{R}(p,q)}}{(1 \oplus \alpha_2)^{m_1}_{\mathcal{R}(p,q)}}\nonumber\\&\times&\genfrac{[}{]}{0pt}{}{n+w_2-1}{w_2}_{\mathcal {R}(p,q)}\frac{\alpha^{w_2}_2\phi^{n-w_2\choose 2}_1\phi^{w_2\choose 2}_2}{(1 \oplus \alpha_2)^{n+w_2}_{\mathcal{R}(p,q)}}.
		\end{eqnarray*}
	\end{small}
	Using the relations
	\begin{eqnarray*}
		[w_2]_{m_2,\mathcal{R}(p^{-1},q^{-1})}=\big(\phi_1\phi_2\big)^{-w_2\,m_2+{m_2+1\choose 2}}\,[w_2]_{m_2,\mathcal{R}(p,q)},
	\end{eqnarray*}
	\begin{eqnarray*}
		{w_2 \choose 2} + {m_2 +1\choose 2}-w_2\,m_2 = {w_2-m_2\choose 2}
	\end{eqnarray*}
	and
	\begin{eqnarray*}
		[w_2]_{m_2,\mathcal{R}(p,q)}[n+w_2+m_1-1]_{m_1,\mathcal{R}(p,q)} \genfrac{[}{]}{0pt}{}{n+w_2-1}{w_2}_{\mathcal {R}(p,q)}&=&[n+m_1+m_2-1]_{m_1+m_2,\mathcal{R}(p,q)}\nonumber\\&\times& \genfrac{[}{]}{0pt}{}{n+m_1+w_2-1}{w_2-m_2}_{\mathcal {R}(p,q)},
	\end{eqnarray*}
	together with the negative $\mathcal{R}(p,q)$-binomial formula \cite{HMRC}, we get
	\begin{small}
		\begin{eqnarray*}
			E\bigg(\frac{[W_2]_{m_2,\mathcal{R}(p^{-1},q^{-1})}[n+W_2+m_1-1]_{m_1,\mathcal{R}(p,q)}}{\big(\phi^{n+W_2}_1 \oplus \alpha_2\,\phi^{n+W_2}_2\big)^{m_1}_{\mathcal{R}(p,q)}}\bigg)=[n+m_2+m_1-1]_{m_1+m_2,\mathcal{R}(p,q)}\,\alpha^{m_2}_2
		\end{eqnarray*}
	\end{small}
	and the proof is achieved.
\end{proof}
\begin{corollary}
	 The $\mathcal{R}(p,q)$- covariance of the $\mathcal{R}(p,q)$- random variables  $\widehat{W}:=(\phi^{n+W_2}_1 \oplus \alpha_2\,\phi^{n+W_2}_2)^{-1}[W_1]_{m_1,\mathcal{R}(p^{-1},q^{-1})}$ and $\overline{W}:=[W_2]_{m_2,\mathcal{R}(p^{-1},q^{-1})}$ is determined by:
	\begin{eqnarray*}
		Cov\big(\widehat{W},\overline{W}\big)=[n]_{\mathcal{R}(p,q)}\alpha_1\alpha_2\big([n+1]_{\mathcal{R}(p,q)}-[n]_{\mathcal{R}(p,q)}\big),
	\end{eqnarray*}
where the $\mathcal{R}(p,q)$-random vector $\underline{W}=\big(W_1,W_2\big)$    follows the negative $\mathcal{R}(p,q)$- trinomial probability distribution of the first kind, with parameters $n,$ $\underline{\alpha}=\big(\alpha_1,\alpha_2\big),$ $p,$ and, $q.$
	\begin{proof}
		By definition, the $\mathcal{R}(p,q)$- covariance of $\widehat{W}$ and $\overline{W}$ is given by
		\begin{eqnarray*}
			Cov\big(\widehat{W},\overline{W}\big)=E\bigg(\widehat{W}\,\overline{W}\bigg)- E\big(\widehat{W}\big)\,E\big(\overline{W}\big).
		\end{eqnarray*}
		Taking $m_1=m_2=1,$ in the relations \eqref{ntfk1},\eqref{ntfk3}, and \eqref{ntfk4}, the result follows. 
	\end{proof}
\end{corollary}
\subsection{$\mathcal{R}(p,q)$-trinomial  distribution of the second kind}
The probability function of the $\mathcal{R}(p,q)$-random vector $\underline{X}=\big(X_1,X_2\big)$ of the $\mathcal{R}(p,q)$- trinomial probability distribution of the second kind, with parameters $n,$ $\underline{\beta}=\big(\beta_1,\beta_2\big),$ $p,$ and, $q$ is given by:
\begin{small}
	\begin{eqnarray}
	P\big(\underline{X}=\underline{x}\big)=\genfrac{[}{]}{0pt}{}{n}{x_1,x_2}_{\mathcal {R}(p,q)}\,\beta^{x_1}_1\,\beta^{x_2}_2(1 \ominus \beta_1)^{n-x_1}_{\mathcal{R}(p,q)}(1 \ominus \beta_2)^{n-x_1-x_2}_{\mathcal{R}(p,q)},
	\end{eqnarray}
\end{small}
where $x_j\in\{0,1,\cdots,n\}, x_1+x_2\leq n, s_j=\displaystyle\sum_{i=1}^{j}x_j, 0<\beta_j<1$, and $j\in\{1,2\}.$
\begin{theorem}
	The $\mathcal{R}(p,q)$-factorial moments of the  $\mathcal{R}(p,q)$-trinomial probability distribution of the second kind, with parameters $n,$ $\underline{\beta}=\big(\beta_1,\beta_2\big),$ $p,$ and, $q,$ are given by:
	\begin{eqnarray}\label{tsk2}
	E\big([X_1]_{m_1,\mathcal{R}(p,q)}\big)=\beta^{m_1}_1\,[n]_{m_1,\mathcal{R}(p,q)}, m_1\in\{0,1,\cdots,n\},
	\end{eqnarray}
	\begin{eqnarray}\label{tsk3}
	E\big([X_2]_{m_2,\mathcal{R}(p,q)}|X_1=x_1\big)=\beta^{m_2}_2\,[n-x_1]_{m_2,\mathcal{R}(p,q)},
	\end{eqnarray}
	where $m_2\in\{0,1,\cdots,n-x_1\},$
	\begin{eqnarray}\label{tsk4}
	E\bigg(\frac{[X_2]_{m_2,\mathcal{R}(p,q)}}{\big(\phi^{n-m_2-X_1}_1 \ominus \beta_1\,\phi^{n-m_2-X_1}_2\big)^{m_2}_{\mathcal{R}(p,q)}}\bigg)=\beta^{m_2}_2\,[n]_{m_2,\mathcal{R}(p,q)}, 
	\end{eqnarray}
	with $ m_2\in\{0,1,\cdots,n\},$
	and
	\begin{small}
	\begin{eqnarray}\label{tsk5}
	E\bigg(\frac{[X_1]_{m_1,\mathcal{R}(p,q)}[X_2]_{m_2,\mathcal{R}(p,q)}}{\big(\phi^{n-m_2-X_1}_1 \ominus \beta_1\,\phi^{n-m_2-X_1}_2\big)^{m_2}_{\mathcal{R}(p,q)}}\bigg)=\beta^{m_1}_1\,\beta^{m_2}_2\,[n]_{m_1+m_2,\mathcal{R}(p,q)},
	\end{eqnarray}
	where $m_1\in\{0,1,\cdots,n-m_2\}$ and $m_2\in\{0,1,\cdots,n\}.$
	\end{small}
\end{theorem}
\begin{proof}
	The $\mathcal{R}(p,q)$- random variable $X_1$ obey the $\mathcal{R}(p,q)$- binomial probability distribution of the second kind,  with mass function:
	\begin{eqnarray*}
		P\big(X_1=x_1\big)=\genfrac{[}{]}{0pt}{}{n}{x_1}_{\mathcal {R}(p,q)}\,\beta^{x_1}_1\,(1 \ominus \beta_1)^{n-x_1}_{\mathcal{R}(p,q)},\,x_1\in\{0,1,\cdots,n\}.
	\end{eqnarray*}
	Thus, using \cite{HMD}, the relation \eqref{tsk2} follows. Furthermore, the conditional distribution of the $\mathcal{R}(p,q)$- random variable $X_2,$ given that $X_1=x_1,$ is a  $\mathcal{R}(p,q)$-binomial probability distribution of the second kind, with mass density 
	\begin{eqnarray*}
		P\big(X_2=x_2|X_1=x_1\big)=\genfrac{[}{]}{0pt}{}{n-x_1}{x_2}_{\mathcal {R}(p,q)}\,\beta^{x_2}_2\,(1 \ominus \beta_2)^{n-x_1-x_2}_{\mathcal{R}(p,q)}.
	\end{eqnarray*}
	According again to \cite{HMD}, the conditional $\mathcal{R}(p,q)$- factorial moments of $X_2,$
	given that $X_1=x_1,$ are furnished by \eqref{tsk3}. The expected value of the $\mathcal{R}(p,q)$-function of $\underline{X}=\big(X_1,X_2\big)$
	\begin{eqnarray*}
		\Delta:=\frac{[X_2]_{m_2,\mathcal{R}(p,q)}}{\big(\phi^{n-m_2-X_1}_1 \ominus \beta_1\,\phi^{n-m_2-X_1}_2\big)^{m_2}_{\mathcal{R}(p,q)}},\,\, m_2\in\{0,1,\cdots,n\}
	\end{eqnarray*}
	can be computed according to the relation
	\begin{eqnarray*}
		E(\Delta)&=& E\big[E(\Delta|X_1)\big]\nonumber\\&=&\beta^{m_2}_2\,E\bigg[\frac{[n-X_1]_{m_2,\mathcal{R}(p,q)}}{\big(\phi^{n-m_2-X_1}_1 \ominus \beta_1\,\phi^{n-m_2-X_1}_2\big)^{m_2}_{\mathcal{R}(p,q)}}\bigg].
	\end{eqnarray*}
	Since
	\begin{small}
		\begin{eqnarray*}
			E\bigg[\frac{[n-X_1]_{m_2,\mathcal{R}(p,q)}}{\big(\phi^{n-m_2-X_1}_1 \ominus \beta_1\,\phi^{n-m_2-X_1}_2\big)^{m_2}_{\mathcal{R}(p,q)}}\bigg]&=&\sum_{x_1=0}^{n-m_2}[n-x_1]_{m_2,\mathcal{R}(p,q)}\genfrac{[}{]}{0pt}{}{n}{x_1}_{\mathcal {R}(p,q)}\nonumber\\&\times&\frac{\beta^{x_1}_1\,(1 \ominus \beta_1)^{n-x_1}_{\mathcal{R}(p,q)}}{\big(\phi^{n-m_2-x_1}_1 \ominus \beta_1\,\phi^{n-m_2-x_1}_2\big)^{m_2}_{\mathcal{R}(p,q)}}.
		\end{eqnarray*}
	\end{small}
	Using the relations:
	\begin{small}
	\begin{eqnarray*}
		[n-x_1]_{m_2,\mathcal{R}(p,q)}\genfrac{[}{]}{0pt}{}{n}{x_1}_{\mathcal {R}(p,q)}=[n]_{m_2,\mathcal{R}(p,q)}\genfrac{[}{]}{0pt}{}{n-m_2}{x_1}_{\mathcal {R}(p,q)},
	\end{eqnarray*}
	\begin{eqnarray}\label{tsk6}
	(1 \ominus \beta_1)^{n-x_1}_{\mathcal{R}(p,q)}=(1 \ominus \beta_1)^{n-m_2-x_1}_{\mathcal{R}(p,q)}\big(\phi^{n-m_2-x_1}_1 \ominus \beta_1\,\phi^{n-m_2-x_1}_2\big)^{m_2}_{\mathcal{R}(p,q)},
	\end{eqnarray}
	\end{small}
	and the $\mathcal{R}(p,q)$-binomial formula \cite{HMRC}, we have:
	\begin{eqnarray*}
		E\bigg[\frac{[n-X_1]_{m_2,\mathcal{R}(p,q)}}{\big(\phi^{n-m_2-X_1}_1 \ominus \beta_1\,\phi^{n-m_2-X_1}_2\big)^{m_2}_{\mathcal{R}(p,q)}}\bigg]&=&[n]_{m_2,\mathcal{R}(p,q)}\sum_{x_1=0}^{n-m_2}\genfrac{[}{]}{0pt}{}{n-m_2}{x_1}_{\mathcal {R}(p,q)}\nonumber\\&\times&\beta^{m_1}_1(1 \ominus \beta_1)^{n-m_2-x_1}_{\mathcal{R}(p,q)}.
	\end{eqnarray*}
	Thus, 
	\begin{eqnarray*}
		E\bigg(\frac{[X_2]_{m_2,\mathcal{R}(p,q)}}{\big(\phi^{n-m_2-X_1}_1 \ominus \beta_1\,\phi^{n-m_2-X_1}_2\big)^{m_2}_{\mathcal{R}(p,q)}}\bigg)=\beta^{m_2}_2\,[n]_{m_2,\mathcal{R}(p,q)}.
	\end{eqnarray*}
	Similarly, the expected value of the $\mathcal{R}(p,q)$-function of $\underline{X}=\big(X_1,X_2\big)$
	\begin{eqnarray*}
		\nabla:=\frac{[X_1]_{m_1,\mathcal{R}(p,q)}[X_2]_{m_2,\mathcal{R}(p,q)}}{\big(\phi^{n-m_2-X_1}_1 \ominus \beta_1\,\phi^{n-m_2-X_1}_2\big)^{m_2}_{\mathcal{R}(p,q)}},\,m_1\in\{0,1,\cdots,n-m_2\},\, m_2\in\{0,1,\cdots,n\}
	\end{eqnarray*}
	may be calculated using the relation:
	\begin{eqnarray*}
		E(\nabla)&=& E\big[E(\nabla|X_1)\big]\nonumber\\&=&\beta^{m_2}_2\,E\bigg[\frac{[X_1]_{m_1,\mathcal{R}(p,q)}[n-X_1]_{m_2,\mathcal{R}(p,q)}}{\big(\phi^{n-m_2-X_1}_1 \ominus \beta_1\,\phi^{n-m_2-X_1}_2\big)^{m_2}_{\mathcal{R}(p,q)}}\bigg].
	\end{eqnarray*}
	Since
	\begin{small}
		\begin{eqnarray*}
			E\bigg[\frac{[X_1]_{m_1,\mathcal{R}(p,q)}[n-X_1]_{m_2,\mathcal{R}(p,q)}}{\big(\phi^{n-m_2-X_1}_1 \ominus \beta_1\,\phi^{n-m_2-X_1}_2\big)^{m_2}_{\mathcal{R}(p,q)}}\bigg]&=&\sum_{x_1=0}^{n-m_2}[x_1]_{m_1,\mathcal{R}(p,q)}[n-x_1]_{m_2,\mathcal{R}(p,q)}\genfrac{[}{]}{0pt}{}{n}{x_1}_{\mathcal {R}(p,q)}\nonumber\\&\times&\frac{\beta^{x_1}_1\,(1 \ominus \beta_1)^{n-x_1}_{\mathcal{R}(p,q)}}{\big(\phi^{n-m_2-x_1}_1 \ominus \beta_1\,\phi^{n-m_2-x_1}_2\big)^{m_2}_{\mathcal{R}(p,q)}}.
		\end{eqnarray*}
	\end{small}
	From  the $\mathcal{R}(p,q)$-binomial formula \cite{HMRC}, the relations \eqref{tsk6}, and 
	\begin{eqnarray*}
		[x_1]_{m_1,\mathcal{R}(p,q)}[n-x_1]_{m_2,\mathcal{R}(p,q)}\genfrac{[}{]}{0pt}{}{n}{x_1}_{\mathcal {R}(p,q)}=[n]_{m_1+m_2,\mathcal{R}(p,q)}\genfrac{[}{]}{0pt}{}{n-m_1-m_2}{x_1-m_1}_{\mathcal {R}(p,q)},
	\end{eqnarray*}
	we obtain:
	\begin{eqnarray*}
		E\bigg[\frac{[X_1]_{m_1,\mathcal{R}(p,q)}[n-X_1]_{m_2,\mathcal{R}(p,q)}}{\big(\phi^{n-m_2-X_1}_1 \ominus \beta_1\,\phi^{n-m_2-X_1}_2\big)^{m_2}_{\mathcal{R}(p,q)}}\bigg]=[n]_{m_1+m_2,\mathcal{R}(p,q)}\beta^{m_1}_1
	\end{eqnarray*}
	and the realtion \eqref{tsk5} follows. 
\end{proof}
\begin{corollary}
	 The $\mathcal{R}(p,q)$- covariance of the functions $\widehat{X}:=[X_1]_{\mathcal{R}(p,q)}$ and  $\overline{X}:=\big(\phi^{n-1-X_1}_1 - \beta_1\phi^{n-1-X_1}_2\big)^{-1}[X_2]_{\mathcal{R}(p,q)}$    is given by:
	\begin{small}
		\begin{eqnarray*}\label{tsk7}
			Cov\big(\widehat{X}, \overline{X}\big)=[n]_{\mathcal{R}(p,q)}\beta_1\beta_2\big([n-1]_{\mathcal{R}(p,q)}-[n]_{\mathcal{R}(p,q)}\big),
		\end{eqnarray*}
		with  $\underline{X}=\big(X_1,X_2\big)$  a  $\mathcal{R}(p,q)$-random vector  satisfying the $\mathcal{R}(p,q)$- trinomial probability distribution of the second kind, with parameters $n,$ $\underline{\beta}=\big(\beta_1,\beta_2\big),$ $p,$ and, $q.$
	\end{small} 
\end{corollary}
\begin{proof}
	Taking $m_1=m_2=1$ in the relations \eqref{tsk2}, \eqref{tsk4}, and \eqref{tsk5}, we compute $E\big(\widehat{X}\, \overline{X}\big)$ and $E\big(\overline{X}\big)E\big(\widehat{X}\big).$
	Therefore, the proof is achieved.
\end{proof}
\subsection{Negative $\mathcal{R}(p,q)$-trinomial  distribution of the second kind}
Let $T_n$ be the number of failures until the occurrence of the $n^{th}$
success, in a sequence of independent geometric sequences of trials. Then, the  distribution of
the $\mathcal{R}(p,q)$- random variable $T_n$ is called negative $\mathcal{R}(p,q)$-binomial distribution of the second kind,
with parameters $n,$ $\beta,$ $p,$ and, $q.$
\begin{lemma}
	The probability function of the negative  $\mathcal{R}(p,q)$- binomial probability distribution of the second kind,  with parameters $n,$ $\beta,$ $p,$ and, $q,$ is given by :
	\begin{eqnarray}\label{ntska}
	P\big(T=t\big)=\genfrac{[}{]}{0pt}{}{n+t-1}{t}_{\mathcal {R}(p,q)}\,\beta^{t}\,(1 \ominus \beta)^{n}_{\mathcal{R}(p,q)},\,t\in\mathbb{N}\cup\{0\}.
	\end{eqnarray}
	Furthermore, its $\mathcal{R}(p,q)$-factorial moments are presented as follows:
	\begin{eqnarray}\label{ntskb}
	E\big([T]_{m,\mathcal{R}(p,q)}\big)=\frac{[n+m-1]_{m,\mathcal{R}(p,q)}\,\beta^{m}}{\big(\phi^{n}_1 \ominus \beta_2\,\phi^{n}_2\big)^{m}_{\mathcal{R}(p,q)}},\,m\in\mathbb{N}.
	\end{eqnarray}
\end{lemma}
Let $\underline{V}=\big(V_1,V_2\big)$ be a  $\mathcal{R}(p,q)$-random vector  obeying the negative $\mathcal{R}(p,q)$- trinomial probability distribution of the second kind, with parameters $n,$ $\underline{\beta}=\big(\beta_1,\beta_2\big),$ $p,$ and, $q.$ Then, its mass function is given by:
\begin{small}
	\begin{eqnarray*}
		P\big(\underline{V}=\underline{v}\big)=\genfrac{[}{]}{0pt}{}{n+v_1+v_2-1}{v_1,v_2}_{\mathcal {R}(p,q)}\beta^{v_1}_1\beta^{v_2}_2(1 \ominus \beta_1)^{n+v_2}_{\mathcal{R}(p,q)}(1 \ominus \beta_2)^{n}_{\mathcal{R}(p,q)},\,v_j\in\mathbb{N}\cup\{0\},
	\end{eqnarray*}
\end{small}
where $ 0<\beta_j<1$, and $j\in\{1,2\}.$
\begin{theorem}
For $m_1\in\mathbb{N}\cup\{0\}$ and $m_2\in\mathbb{N}\cup\{0\},$	the $\mathcal{R}(p,q)$-factorial moments of the negative $\mathcal{R}(p,q)$-trinomial probability distribution of the second kind, with parameters $n,$ $\underline{\beta}=\big(\beta_1,\beta_2\big),$ $p,$ and, $q$ are presented as follows:
	\begin{eqnarray}\label{ntsk2}
	E\big([V_2]_{m_2,\mathcal{R}(p,q)}\big)=\frac{\beta^{m_2}_2\,[n+m_2-1]_{m_2,\mathcal{R}(p,q)}}{\big(\phi^{n}_1 \ominus \beta_2\,\phi^{n}_2\big)^{m_2}_{\mathcal{R}(p,q)}},
	\end{eqnarray}
	\begin{eqnarray}\label{ntsk3}
	E\big([V_1]_{m_,\mathcal{R}(p,q)}|V_2=v_2\big)=\frac{\beta^{m_1}_1\,[n+v_2+m_1-1]_{m_1,\mathcal{R}(p,q)}}{\big(\phi^{n+v_2}_1 \ominus \beta_1\,\phi^{n+v_2}_2\big)^{m_1}_{\mathcal{R}(p,q)}},
	\end{eqnarray}
	\begin{eqnarray}\label{ntsk4}
	E\bigg(\frac{[V_1]_{m_1,\mathcal{R}(p,q)}}{\big(\phi^{n+V_2}_1 \ominus \beta_1\,\phi^{n+V_2}_2\big)^{-m_1}_{\mathcal{R}(p,q)}}\bigg)=\frac{\beta^{m_1}_1\,[n+m_1-1]_{m_1,\mathcal{R}(p,q)}}{\big(\phi^{n}_1 \ominus \beta_2\,\phi^{n}_2\big)^{m_1}_{\mathcal{R}(p,q)}},
	\end{eqnarray}
	and
	\begin{small}
	\begin{equation}\label{ntsk5}
	E\bigg(\frac{[V_1]_{m_1,\mathcal{R}(p,q)}[V_2]_{m_2,\mathcal{R}(p,q)}}{\big(\phi^{n+V_2}_1 \ominus \beta_1\phi^{n+V_2}_2\big)^{-m_2}_{\mathcal{R}(p,q)}}\bigg)=\frac{\beta^{m_1}_1\beta^{m_2}_2[n+m_1+m_2-1]_{m_1+m_2,\mathcal{R}(p,q)}}{\big(\phi^{n}_1 \ominus \beta_2\phi^{n}_2\big)^{m_1+m_2}_{\mathcal{R}(p,q)}}.
	\end{equation}
	\end{small}
\end{theorem}
\begin{proof}
	The relation \eqref{ntsk2} comes from \eqref{ntska} and \eqref{ntskb}.  Furthermore, the conditional distribution of the $\mathcal{R}(p,q)$- random variable $V_1,$ given that $V_2=v_2,$ is a negative $\mathcal{R}(p,q)$-binomial  distribution of the second kind, with probability function: 
	\begin{eqnarray*}
		P\big(V_1=v_1|V_2=v_2\big)=\genfrac{[}{]}{0pt}{}{n+v_1+v_2-1}{v_1}_{\mathcal {R}(p,q)}\,\beta^{v_1}_1\,(1 \ominus \beta_1)^{n+v_2}_{\mathcal{R}(p,q)},\,v\in\mathbb{N}\cup\{0\}.
	\end{eqnarray*}
	Using once the relations \eqref{ntska} and \eqref{ntskb}, the conditional $\mathcal{R}(p,q)$- factorial moments of $V_1,$
	given that $V_2=v_2,$ are given by \eqref{ntsk3}. Besides, the expected value of the $\mathcal{R}(p,q)$-function of $\underline{V}=\big(V_1,V_2\big)$
	\begin{eqnarray*}
		\widehat{\Delta}:=\frac{[V_1]_{m_1,\mathcal{R}(p,q)}}{\big(\phi^{n+V_2}_1 \ominus \beta_2\,\phi^{n+V_2}_2\big)^{m_1}_{\mathcal{R}(p,q)}},\,\, m_1\in\mathbb{N}\cup\{0\}
	\end{eqnarray*}
	may be evaluated according to the relation:
	\begin{eqnarray*}
		E(\widehat{\Delta})&=& E\big[E(\widehat{\Delta}|V_2)\big]\nonumber\\&=&\beta^{m_1}_1\,E\bigg([n+V_2+m_1-1]_{m_1,\mathcal{R}(p,q)}\bigg).
	\end{eqnarray*}
	Since
	\begin{small}
		\begin{eqnarray*}
			E\bigg([n+V_2+m_1-1]_{m_1,\mathcal{R}(p,q)}\bigg)&=&\sum_{v_2=0}^{\infty}[n+v_2+m_1-1]_{m_1,\mathcal{R}(p,q)}\nonumber\\&\times&\genfrac{[}{]}{0pt}{}{n+v_2-1}{v_2}_{\mathcal {R}(p,q)}\beta^{m_2}_2\,(1 \ominus \beta_2)^{n}_{\mathcal{R}(p,q)}.
		\end{eqnarray*}
	\end{small}
	From the relation
	\begin{eqnarray*}
		[n+v_2+m_1-1]_{m_1,\mathcal{R}(p,q)}\genfrac{[}{]}{0pt}{}{n+v_2-1}{v_2}_{\mathcal {R}(p,q)}=[n+m_1-1]_{m_1,\mathcal{R}(p,q)}\genfrac{[}{]}{0pt}{}{n+m_1+v_2-1}{v_2}_{\mathcal {R}(p,q)},
	\end{eqnarray*}
	and the negative  $\mathcal{R}(p,q)$-binomial formula \cite{HMRC}, we have:
	\begin{eqnarray*}
		E\bigg([n+V_2+m_1-1]_{m_1,\mathcal{R}(p,q)}\bigg)&=&[n+m_1-1]_{m_1,\mathcal{R}(p,q)}(1 \ominus \beta_2)^{n}_{\mathcal{R}(p,q)}\sum_{v_2=0}^{\infty}\nonumber\\&\times&\genfrac{[}{]}{0pt}{}{n+v_2-1}{v_2}_{\mathcal {R}(p,q)}\beta^{m_2}_2\nonumber\\&=& \frac{[n+m_1-1]_{m_1,\mathcal{R}(p,q)}}{(\phi^{n}_1 \ominus \beta_2\phi^{n}_2)^{m_1}_{\mathcal{R}(p,q)}}.
	\end{eqnarray*}
	Thus, the relation \eqref{ntsk4} holds.
	
	So, the expected value of the $\mathcal{R}(p,q)$-function of $\underline{V}=\big(V_1,V_2\big)$
	\begin{eqnarray*}
		\widehat{\nabla}:=\frac{[V_1]_{m_1,\mathcal{R}(p,q)}[V_2]_{m_2,\mathcal{R}(p,q)}}{\big(\phi^{n+V_2}_1 \ominus \beta_1\,\phi^{n+V_2}_2\big)^{-m_1}_{\mathcal{R}(p,q)}},\,m_1\in\mathbb{N}\cup\{0\},\, m_2\in\mathbb{N}\cup\{0\}
	\end{eqnarray*}
	can be calculated using the relation:
	\begin{eqnarray*}
		E(\widehat{\nabla})&=& E\big[E(\widehat{\nabla}|V_2)\big]\nonumber\\&=&\beta^{m_1}_1\,E\bigg[[V_2]_{m_2,\mathcal{R}(p,q)}[n+V_2+m_1-1]_{m_1,\mathcal{R}(p,q)}\bigg].
	\end{eqnarray*}
	Since
	\begin{small}
		\begin{eqnarray*}
			E\bigg[[V_2]_{m_2,\mathcal{R}(p,q)}[n+V_2+m_1-1]_{m_1,\mathcal{R}(p,q)}\bigg]&=&\sum_{v_2=m_2=0}^{\infty}[v_2]_{m_2,\mathcal{R}(p,q)}[n+v_2+m_1-1]_{m_1,\mathcal{R}(p,q)}\nonumber\\&\times&\genfrac{[}{]}{0pt}{}{n+v_2-1}{v_2}_{\mathcal {R}(p,q)}\beta^{m_2}_2\,(1 \ominus \beta_2)^{n}_{\mathcal{R}(p,q)}.
		\end{eqnarray*}
	\end{small}
	From  the $\mathcal{R}(p,q)$-binomial formula \cite{HMRC}, the relations \eqref{tsk6}, and 
	\begin{eqnarray*}
		[v_2]_{m_2,\mathcal{R}(p,q)}[n+v_2+m_1-1]_{m_1,\mathcal{R}(p,q)}\genfrac{[}{]}{0pt}{}{n+v_2-1}{v_2}_{\mathcal {R}(p,q)}&=&[n+m_1+m_2-1]_{m_1+m_2,\mathcal{R}(p,q)}\nonumber\\&\times&\genfrac{[}{]}{0pt}{}{n+m_1+v_2-1}{v_2-m_2}_{\mathcal {R}(p,q)},
	\end{eqnarray*}
	we obtain:
	\begin{eqnarray*}
		E\bigg[[V_2]_{m_2,\mathcal{R}(p,q)}[n+V_2+m_1-1]_{m_1,\mathcal{R}(p,q)}\bigg]=\frac{[n+m_1+m_2-1]_{m_1+m_2,\mathcal{R}(p,q)}\,\beta^{m_2}_2}{\big(\phi^{n}_1 \ominus \beta_2\,\phi^{n}_2\big)^{m_1+m_2}_{\mathcal{R}(p,q)}}
	\end{eqnarray*}
	and the realtion \eqref{ntsk5} follows.
\end{proof}
\begin{corollary}
	The $\mathcal{R}(p,q)$-covariance of the functions  and $\widehat{V}:=(\phi^{n+V_2}_1 \ominus \beta_2\,\phi^{n+V_2}_2)[V_1]_{m_2,\mathcal{R}(p,q)}$ and $\overline{V}:=[V_2]_{m_1,\mathcal{R}(p,q)}$ is given by:
	\begin{eqnarray}
	Cov\big(\widehat{V},\overline{V}\big)=\frac{[n]_{\mathcal{R}(p,q)}\beta_1\beta_2}{(\phi^{n}_1 - \beta_2\phi^{n}_2)}\bigg(\frac{[n+1]_{\mathcal{R}(p,q)}}{(\phi^{n+1}_1 - \beta_2\phi^{n+1}_2)}-\frac{[n]_{\mathcal{R}(p,q)}}{(\phi^{n}_1- \beta_2\phi^{n}_2)}\bigg),
	\end{eqnarray}
	where  $\underline{V}=\big(V_1,V_2\big)$ is a  $\mathcal{R}(p,q)$-random vector  satisfying the negative $\mathcal{R}(p,q)$- trinomial probability distribution of the second kind, with parameters $n,$ $\underline{\beta}=\big(\beta_1,\beta_2\big),$ $p,$ and, $q.$ 
\end{corollary}
\begin{proof}
	Putting $m_1=m_2=1,$ in the relations \eqref{ntsk2}, \eqref{ntsk4}, and \eqref{ntsk5}, we obtain:
	\begin{small}
		\begin{eqnarray*}
			E\big(\widehat{V}\overline{V}\big)=\frac{[n+1]_{\mathcal{R}(p,q)}[n]_{\mathcal{R}(p,q)}\beta_1\beta_2}{(\phi^{n}_1 - \beta_2\phi^{n}_2)(\phi^{n+1}_1 - \beta_2\phi^{n+1}_2)}\,\,\mbox{and}\,\,E\big(\widehat{V}\big)E\big(\overline{V}\big)=\frac{[n]_{\mathcal{R}(p,q)}[n]_{\mathcal{R}(p,q)}\beta_2\beta_1}{(\phi^{n}_1- \beta_2\phi^{n}_2)^2}.
		\end{eqnarray*}
	\end{small}
	After computation, the result follows. 
\end{proof}
\section{Trinomial distributions and particular quantum algebras}
In this section, we deduce the $\mathcal{R}(p,q)$-trinomial probability distribution of the first and second kind, their negative and properties from some quantum algebra existing in the literature.
\subsection{Trinomial distribution and Biedenharn-Macfarlane algebra \cite{BC, M}}
Taking $\mathcal{R}(x)=\frac{x-x^{-1}}{q-q^{-1}},$ we obtain:
\begin{enumerate}
	\item[(i)]  The $q$-trinomial  distribution of the first kind
	is given by
	\begin{eqnarray}
	P\big(Y_1=y_1,Y_2=y_2\big)=\genfrac{[}{]}{0pt}{}{n}{y_1,y_2}_{q}\frac{\alpha^{y_1}_1\alpha^{y_2}_2\,q^{{n-y_1\choose 2}+{n-y_2\choose 2}}\,q^{-{y_1\choose 2}-{y_2\choose 2}}}{(1 \oplus \alpha_1)^{n}_{q}(1 \oplus \alpha_2)^{n-y_1}_{q}}
	\end{eqnarray}
	and the $q$-factorial moments are:
	\begin{small}
		\begin{eqnarray*}
			E\big([Y_1]_{m_1,q}\big)=\frac{[n]_{m_1,q}\,\alpha^{m_1}_1\,q^{-{m_1\choose 2}}}{(1 \oplus \alpha_1)^{m_1}_{q}},\,m_1\in\{0,1,\cdots,n\},\\
			E\big([Y_2]_{m_2,q}|Y_1=y_1\big)=\frac{[n-y_1]_{m_2,q}\,\alpha^{m_2}_2\,q^{-{m_2\choose 2}}}{(1 \oplus \alpha_2)^{m_2}_{q}},\,m_2\in\{0,1,\cdots,n-y_1\},\\
			E\big([Y_2]_{m_2,p,q}\big)=\frac{[n]_{m_2,p,q}\,\alpha^{m_2}_2\,q^{m_2\choose 2}\,p^{{m_2\choose 2}+m_2(n-m_2)}}{(1 \oplus \alpha_2)^{m_2}_{p,q}(p^{n-m_2} \oplus \alpha_1\,q^{n-m_2})^{m_2}_{p,q}},\,m_2\in\{0,1,\cdots,n\},
		\end{eqnarray*}
		and 
		\begin{eqnarray*}
			E\big([Y_1]_{m_1,
				q}[Y_2]_{m_2,q}\big)
			=\frac{[n]_{m_1+m_2,q}\alpha^{m_1}_1\alpha^{m_2}_2\,q^{{m_2\choose 2}+m_2(n-m_2)}q^{-{m_1\choose 2}-{m_2 \choose 2}}}{(1 \oplus \alpha_1)^{m_1}_{q}(1 \oplus \alpha_2)^{m_2}_{q}(p^{n-m_2} \oplus \alpha_1\,q^{n-m_2})^{m_2}_{q}},
		\end{eqnarray*}
		where $m_1\in\{0,1,\cdots, n-m_2\}$ and $m_2\in\{0,1,\cdots, n\}.$ Moreover, the  $q$- covariance is derived as follows:
		\begin{eqnarray*}
			Cov\big([Y_1]_{q},[Y_2]_{q}\big)
			&=&\frac{q^{n-1}\alpha_1\,\alpha_2\,[n]_{q}\big([n-1]_{q}-[n]_{q}\big)}{(1 + \alpha_1)(1 + \alpha_2)(q^{n-1} + \alpha_1\,q^{-n+1})}.
		\end{eqnarray*}
	\end{small}
	\item[(ii)] The negative  $q$-trinomial  distribution of the first kind, with parameters $n,$ $\underline{\alpha}=\big(\alpha_1,\alpha_2\big),$  and, $q$ is presented by:
	\begin{small}
		\begin{eqnarray*}
			P\big(\underline{W}=\underline{w}\big)=\genfrac{[}{]}{0pt}{}{n+w_1+w_2-1}{w_1,w_2}_{q}\frac{\alpha^{w_1}_1\alpha^{w_2}_2\,q^{{n-w_1\choose 2}+{n-w_2\choose 2}}\,q^{-{w_1\choose 2}-{w_2\choose 2}}}{(1 \oplus \alpha_1)^{n+w_1+w_2}_{q}(1 \oplus \alpha_2)^{n+w_2}_{q}},
		\end{eqnarray*}
	\end{small}
	where $w_j\in\mathbb{N}\cup\{0\}, 0<\alpha_j<1$, and $j\in\{1,2\}.$
	Besides, for $m_1\in\mathbb{N}\cup\{0\}$ and $m_2\in\mathbb{N}\cup\{0\},$ 
	
	its $q^{-1}$-factorial moments  are given as follows:
	\begin{small}
		\begin{eqnarray*}
			E\big([W_2]_{m_2,q^{-1}}\big)=[n+m_2-1]_{m_2,q}\,\alpha^{m_2}_2,\,m_2\in\mathbb{N}\cup\{0\},
		\end{eqnarray*}
		\begin{eqnarray*}
			E\big([W_1]_{m_1,q^{-1}}|W_2=w_2\big)=[n+w_2+m_1-1]_{m_1,q}\,\alpha^{m_1}_1,\,m_1\in\mathbb{N}\cup\{0\},
		\end{eqnarray*}
		\begin{eqnarray*}
			E\bigg(\frac{[W_1]_{m_1,q^{-1}}}{\big(p^{n+W_2} \oplus \alpha_2\,q^{-n-W_2}\big)^{m_1}_{q}}\bigg)=[n+m_1-1]_{m_1,q}\,\alpha^{m_1}_1,\,m_1\in\mathbb{N}\cup\{0\},
		\end{eqnarray*}
		and 
		\begin{eqnarray*}
			E\bigg(\frac{[W_1]_{m_1,q^{-1}}\,[W_2]_{m_2,q^{-1}}}{\big(p^{n+W_2} \oplus \alpha_2\,q^{n+W_2}\big)^{m_1}_{q}}\bigg)
			=[n+m_1+m_2-1]_{m_1+m_2,q}\,\alpha^{m_1}_1\alpha^{m_2}_2,
		\end{eqnarray*}
	\end{small}
	Furthermore, 
	the covariance of $\widehat{W}:=(q^{n+W_2} \oplus \alpha_2\,q^{-n-W_2})^{-1}[W_1]_{m_1,q^{-1}}$ and $\overline{W}:=[W_2]_{m_2,q^{-1}}$ is determined by:
	\begin{eqnarray*}
		Cov\big(\widehat{W},\overline{W}\big)=[n]_{q}\alpha_1\alpha_2\big([n+1]_{q}-[n]_{q}\big).
	\end{eqnarray*}
	\item[(iii)]  The $q$-trinomial distribution of the second kind, with parameters $n,$ $\underline{\beta}=\big(\beta_1,\beta_2\big),$  and, $q$ is given by:
	\begin{small}
		\begin{eqnarray*}
			P\big(X_1=x_1,X_2=x_2\big)=\genfrac{[}{]}{0pt}{}{n}{x_1,x_2}_{q}\,\beta^{x_1}_1\,\beta^{x_2}_2(1 \ominus \beta_1)^{n-x_1}_{q}(1 \ominus \beta_2)^{n-x_1-x_2}_{q},
		\end{eqnarray*}
	\end{small}
	where $x_j\in\{0,1,\cdots,n\}, x_1+x_2\leq n, s_j=\displaystyle\sum_{i=1}^{j}x_j, 0<\beta_j<1$, and $j\in\{1,2\}.$
	Besides, its 
	$q$-factorial moments  are given by:
	\begin{eqnarray*}
		E\big([X_1]_{m_1,q}\big)=[n]_{m_1,q}\,\beta^{m_1}_1,\,m_1\in\{0,1,\cdots,n\},
	\end{eqnarray*}
	\begin{eqnarray*}
		E\big([X_2]_{m_2,q}|X_1=x_1\big)=[n-x_1]_{m_2,q}\,\beta^{m_2}_2,\,m_2\in\{0,1,\cdots,n-x_1\},
	\end{eqnarray*}
	\begin{eqnarray*}
		E\bigg(\frac{[X_2]_{m_2,q}}{\big(p^{n-m_2-X_1} \ominus \beta_1\,q^{n-m_2-X_1}\big)^{m_2}_{q}}\bigg)=[n]_{m_2,q}\,\beta^{m_2}_2, \, m_2\in\{0,1,\cdots,n\},
	\end{eqnarray*}
	and
	\begin{eqnarray*}
		E\bigg(\frac{[X_1]_{m_1,q}[X_2]_{m_2,q}}{\big(p^{n-m_2-X_1} \ominus \beta_1\,q^{n-m_2-X_1}\big)^{m_2}_{q}}\bigg)=[n]_{m_1+m_2,q}\,\beta^{m_1}_1\,\beta^{m_2}_2,
	\end{eqnarray*}
	where $m_1\in\{0,1,\cdots,n-m_2\}$ and $m_2\in\{0,1,\cdots,n\}.$ Moreover, 
	the  $q$- covariance of the functions  $\big(q^{n-1-X_1} - \beta_1\,q^{-n+1+X_1}\big)^{-1}[X_2]_{q}$ and $[X_1]_{q}$  is given by:
	\begin{small}
		\begin{eqnarray*}
			Cov\bigg[[X_1]_{q}, \big(q^{n-1-X_1} - \beta_1\,q^{-n+1+X_1}\big)^{-1}[X_2]_{q}\bigg]=[n]_{q}\beta_1\beta_2\big([n-1]_{q}-[n]_{q}\big).
		\end{eqnarray*}
	\end{small} 
	\item[(iii)]
	The negative $q$- trinomial probability distribution of the second kind, with parameters $n,$ $\underline{\beta}=\big(\beta_1,\beta_2\big),$  and, $q$  is given by:
	\begin{small}
		\begin{eqnarray*}
			P\big(\underline{V}=\underline{v}\big)=\genfrac{[}{]}{0pt}{}{n+v_1+v_2-1}{v_1,v_2}_{q}\beta^{v_1}_1\beta^{v_2}_2(1 \ominus \beta_1)^{n+v_2}_{q}(1 \ominus \beta_2)^{n}_{q},\,v_j\in\mathbb{N}\cup\{0\},
		\end{eqnarray*}
	\end{small}
	where $ 0<\beta_j<1$, and $j\in\{1,2\}.$
	Moreover, 
	its $q$-factorial moments of  are presented as follows:
	\begin{eqnarray*}
		E\big([V_2]_{m_2,q}\big)=\frac{[n+m_2-1]_{m_2,q}\,\beta^{m_2}_2}{\big(p^{n} \ominus \beta_2\,q^{n}\big)^{m_2}_{q}},
	\end{eqnarray*}
	\begin{eqnarray*}
		E\big([V_1]_{m_,q}|V_2=v_2\big)=\frac{[n+v_2+m_1-1]_{m_1,q}\,\beta^{m_1}_1}{\big(p^{n+v_2} \ominus \beta_1\,q^{n+v_2}\big)^{m_1}_{q}},
	\end{eqnarray*}
	\begin{eqnarray*}
		E\bigg(\frac{[V_1]_{m_1,q}}{\big(p^{n+V_2} \ominus \beta_1\,q^{n+V_2}\big)^{-m_1}_{q}}\bigg)=\frac{[n+m_1-1]_{m_1,q}\,\beta^{m_1}_1}{\big(p^{n} \ominus \beta_2\,q^{n}\big)^{m_1}_{q}},
	\end{eqnarray*}
	and
	\begin{eqnarray*}
		E\bigg(\frac{[V_1]_{m_1,q}[V_2]_{m_2,q}}{\big(p^{n+V_2} \ominus \beta_1\,q^{n+V_2}\big)^{-m_2}_{q}}\bigg)=\frac{[n+m_1+m_2-1]_{m_1+m_2,q}\,\beta^{m_1}_1\,\beta^{m_2}_2}{\big(p^{n} \ominus \beta_2\,q^{n}\big)^{m_1+m_2}_{q}},
	\end{eqnarray*}
	where $m_1\in\mathbb{N}\cup\{0\}$ and $m_2\in\mathbb{N}\cup\{0\}.$
	Furthermore,   the $q$-covariance of the functions  and $\widehat{V}:=(q^{n+V_2} - \beta_2\,q^{-n-V_2})[V_1]_{m_2,q}$ and $\overline{V}:=[V_2]_{m_1,q}$ is given by:
	\begin{eqnarray*}
		Cov\big(\widehat{V},\overline{V}\big)=\frac{[n]_{q}\beta_1\beta_2}{(q^{n} - \beta_2\,q^{-n})}\bigg(\frac{[n+1]_{q}}{(q^{n+1} - \beta_2\,q^{-n-1})}-\frac{[n]_{q}}{(q^{n}- \beta_2\,q^{-n})}\bigg).
	\end{eqnarray*}
\end{enumerate}
\subsection{Trinomial distribution and Jagannathan-Srinivasa algebra \cite{JS}}
 For illustration, we investigate    by taking $\mathcal{R}(x,y)=(p-q)^{-1}(x-y):$
	\begin{enumerate}
		\item[(i)]  The $(p,q)$-trinomial  distribution of the first kind
		is given by
		\begin{eqnarray}\label{pqmfk}
		P\big(Y_1=y_1,Y_2=y_2\big)=\genfrac{[}{]}{0pt}{}{n}{y_1,y_2}_{p,q}\frac{\alpha^{y_1}_1\alpha^{y_2}_2\,p^{{n-y_1\choose 2}+{n-y_2\choose 2}}\,q^{{y_1\choose 2}+{y_2\choose 2}}}{(1 \oplus \alpha_1)^{n}_{p,q}(1 \oplus \alpha_2)^{n-y_1}_{p,q}}
		\end{eqnarray}
		and the $(p,q)$-factorial moments are:
		\begin{small}
			\begin{eqnarray*}
				E\big([Y_1]_{m_1,p,q}\big)=\frac{[n]_{m_1,p,q}\,\alpha^{m_1}_1\,q^{m_1\choose 2}}{(1 \oplus \alpha_1)^{m_1}_{p,q}},\,m_1\in\{0,1,\cdots,n\},\\
				E\big([Y_2]_{m_2,p,q}|Y_1=y_1\big)=\frac{[n-y_1]_{m_2,p,q}\,\alpha^{m_2}_2\,q^{m_2\choose 2}}{(1 \oplus \alpha_2)^{m_2}_{p,q}},\,m_2\in\{0,1,\cdots,n-y_1\},\\
				E\big([Y_2]_{m_2,p,q}\big)=\frac{[n]_{m_2,p,q}\,\alpha^{m_2}_2\,q^{m_2\choose 2}\,p^{{m_2\choose 2}+m_2(n-m_2)}}{(1 \oplus \alpha_2)^{m_2}_{p,q}(p^{n-m_2} \oplus \alpha_1\,q^{n-m_2})^{m_2}_{p,q}},\,m_2\in\{0,1,\cdots,n\},
			\end{eqnarray*}
			and 
			\begin{eqnarray*}
				E\big([Y_1]_{m_1,p,q}[Y_2]_{m_2,p,q}\big)
				=\frac{[n]_{m_1+m_2,p,q}\alpha^{m_1}_1\alpha^{m_2}_2\,p^{{m_2\choose 2}+m_2(n-m_2)}q^{{m_1\choose 2}+{m_2 \choose 2}}}{(1 \oplus \alpha_1)^{m_1}_{p,q}(1 \oplus \alpha_2)^{m_2}_{p,q}(p^{n-m_2} \oplus \alpha_1\,q^{n-m_2})^{m_2}_{p,q}},
			\end{eqnarray*}
			where $m_1\in\{0,1,\cdots, n-m_2\}$ and $m_2\in\{0,1,\cdots, n\}.$ Moreover, the  $(p,q)$- covariance is derived as follows:
			\begin{eqnarray*}
				Cov\big([Y_1]_{p,q},[Y_2]_{\mathcal{R}(p,q)}\big)
				&=&\frac{p^{n-1}\alpha_1\,\alpha_2\,[n]_{p,q}\big([n-1]_{p,q}-[n]_{p,q}\big)}{(1 + \alpha_1)(1 + \alpha_2)(p^{n-1} + \alpha_1\,q^{n-1})}.
			\end{eqnarray*}
		\end{small}
		\item[(ii)] The negative  $(p,q)$-trinomial  distribution of the first kind, with parameters $n,$ $\underline{\alpha}=\big(\alpha_1,\alpha_2\big),$ $p,$ and, $q$ is presented by:
		\begin{small}
			\begin{eqnarray*}
				P\big(\underline{W}=\underline{w}\big)=\genfrac{[}{]}{0pt}{}{n+w_1+w_2-1}{w_1,w_2}_{p,q}\frac{\alpha^{w_1}_1\alpha^{w_2}_2\,p^{{n-w_1\choose 2}+{n-w_2\choose 2}}\,q^{{w_1\choose 2}+{w_2\choose 2}}}{(1 \oplus \alpha_1)^{n+w_1+w_2}_{p,q}(1 \oplus \alpha_2)^{n+w_2}_{p,q}},
			\end{eqnarray*}
		\end{small}
		where $w_j\in\mathbb{N}\cup\{0\}, 0<\alpha_j<1$, and $j\in\{1,2\}.$
		Besides, for $m_1\in\mathbb{N}\cup\{0\}$ and $m_2\in\mathbb{N}\cup\{0\},$ 
		
		its $(p^{-1},q^{-1})$-factorial moments  are given as follows:
		\begin{small}
			\begin{eqnarray*}
				E\big([W_2]_{m_2,p^{-1},q^{-1}}\big)=[n+m_2-1]_{m_2,p,q}\,\alpha^{m_2}_2,\,m_2\in\mathbb{N}\cup\{0\},
			\end{eqnarray*}
			\begin{eqnarray*}
				E\big([W_1]_{m_1,p^{-1},q^{-1}}|W_2=w_2\big)=[n+w_2+m_1-1]_{m_1,p,q}\,\alpha^{m_1}_1,\,m_1\in\mathbb{N}\cup\{0\},
			\end{eqnarray*}
			\begin{eqnarray*}
				E\bigg(\frac{[W_1]_{m_1,p^{-1},q^{-1}}}{\big(p^{n+W_2} \oplus \alpha_2\,q^{n+W_2}\big)^{m_1}_{p,q}}\bigg)=[n+m_1-1]_{m_1,p,q}\,\alpha^{m_1}_1,\,m_1\in\mathbb{N}\cup\{0\},
			\end{eqnarray*}
			and 
				\begin{eqnarray*}
					E\bigg(\frac{[W_1]_{m_1,p^{-1},q^{-1}}\,[W_2]_{m_2,p^{-1},q^{-1}}}{\big(p^{n+W_2} \oplus \alpha_2\,q^{n+W_2}\big)^{m_1}_{p,q}}\bigg)
					=[n+m_1+m_2-1]_{m_1+m_2,p,q}\,\alpha^{m_1}_1\alpha^{m_2}_2,
				\end{eqnarray*}
			\end{small}
		Furthermore, 
		the covariance of $\widehat{W}:=(p^{n+W_2} \oplus \alpha_2\,q^{n+W_2})^{-1}[W_1]_{m_1,p^{-1},q^{-1}}$ and $\overline{W}:=[W_2]_{m_2,p^{-1},q^{-1}}$ is determined by:
		\begin{eqnarray*}
			Cov\big(\widehat{W},\overline{W}\big)=[n]_{p,q}\alpha_1\alpha_2\big([n+1]_{p,q}-[n]_{p,q}\big).
		\end{eqnarray*}
		\item[(iii)]  The $(p,q)$-trinomial distribution of the second kind, with parameters $n,$ $\underline{\beta}=\big(\beta_1,\beta_2\big),$ $p,$ and, $q$ is given by:
		\begin{small}
			\begin{eqnarray*}
				P\big(X_1=x_1,X_2=x_2\big)=\genfrac{[}{]}{0pt}{}{n}{x_1,x_2}_{p,q}\,\beta^{x_1}_1\,\beta^{x_2}_2(1 \ominus \beta_1)^{n-x_1}_{p,q}(1 \ominus \beta_2)^{n-x_1-x_2}_{p,q},
			\end{eqnarray*}
		\end{small}
		where $x_j\in\{0,1,\cdots,n\}, x_1+x_2\leq n, s_j=\displaystyle\sum_{i=1}^{j}x_j, 0<\beta_j<1$, and $j\in\{1,2\}.$
		Besides, its 
		$(p,q)$-factorial moments  are given by:
		\begin{eqnarray*}
			E\big([X_1]_{m_1,p,q}\big)=[n]_{m_1,p,q}\,\beta^{m_1}_1,\,m_1\in\{0,1,\cdots,n\},
		\end{eqnarray*}
		\begin{eqnarray*}
			E\big([X_2]_{m_2,p,q}|X_1=x_1\big)=[n-x_1]_{m_2,p,q}\,\beta^{m_2}_2,\,m_2\in\{0,1,\cdots,n-x_1\},
		\end{eqnarray*}
		\begin{eqnarray*}
			E\bigg(\frac{[X_2]_{m_2,p,q}}{\big(p^{n-m_2-X_1} \ominus \beta_1\,q^{n-m_2-X_1}\big)^{m_2}_{p,q}}\bigg)=[n]_{m_2,p,q}\,\beta^{m_2}_2, \, m_2\in\{0,1,\cdots,n\},
		\end{eqnarray*}
		and
		\begin{eqnarray*}
			E\bigg(\frac{[X_1]_{m_1,p,q}[X_2]_{m_2,p,q}}{\big(p^{n-m_2-X_1} \ominus \beta_1\,q^{n-m_2-X_1}\big)^{m_2}_{p,q}}\bigg)=[n]_{m_1+m_2,p,q}\,\beta^{m_1}_1\,\beta^{m_2}_2,
		\end{eqnarray*}
		where $m_1\in\{0,1,\cdots,n-m_2\}$ and $m_2\in\{0,1,\cdots,n\}.$ Moreover, 
		the  $(p,q)$- covariance of the functions  $\big(p^{n-1-X_1} - \beta_1\,q^{n-1-X_1}\big)^{-1}[X_2]_{p,q}$ and $[X_1]_{p,q}$  is given by:
		\begin{small}
			\begin{eqnarray*}
				Cov\bigg[[X_1]_{p,q}, \big(p^{n-1-X_1} - \beta_1\,q^{n-1-X_1}\big)^{-1}[X_2]_{p,q}\bigg]=[n]_{p,q}\beta_1\beta_2\big([n-1]_{p,q}-[n]_{p,q}\big).
			\end{eqnarray*}
		\end{small} 
		\item[(iii)]
		The negative $(p,q)$- trinomial probability distribution of the second kind, with parameters $n,$ $\underline{\beta}=\big(\beta_1,\beta_2\big),$ $p,$ and, $q$  is given by:
		\begin{small}
			\begin{eqnarray*}
				P\big(\underline{V}=\underline{v}\big)=\genfrac{[}{]}{0pt}{}{n+v_1+v_2-1}{v_1,v_2}_{p,q}\beta^{v_1}_1\beta^{v_2}_2(1 \ominus \beta_1)^{n+v_2}_{p,q}(1 \ominus \beta_2)^{n}_{p,q},\,v_j\in\mathbb{N}\cup\{0\},
			\end{eqnarray*}
		\end{small}
		where $ 0<\beta_j<1$, and $j\in\{1,2\}.$
		Moreover, 
		its $(p,q)$-factorial moments of  are presented as follows:
		\begin{eqnarray*}
			E\big([V_2]_{m_2,p,q}\big)=\frac{[n+m_2-1]_{m_2,p,q}\,\beta^{m_2}_2}{\big(p^{n} \ominus \beta_2\,q^{n}\big)^{m_2}_{p,q}},
		\end{eqnarray*}
		\begin{eqnarray*}
			E\big([V_1]_{m_,p,q}|V_2=v_2\big)=\frac{[n+v_2+m_1-1]_{m_1,p,q}\,\beta^{m_1}_1}{\big(p^{n+v_2} \ominus \beta_1\,q^{n+v_2}\big)^{m_1}_{p,q}},
		\end{eqnarray*}
		\begin{eqnarray*}
			E\bigg(\frac{[V_1]_{m_1,p,q}}{\big(p^{n+V_2} \ominus \beta_1\,q^{n+V_2}\big)^{-m_1}_{p,q}}\bigg)=\frac{[n+m_1-1]_{m_1,p,q}\,\beta^{m_1}_1}{\big(p^{n} \ominus \beta_2\,q^{n}\big)^{m_1}_{p,q}},
		\end{eqnarray*}
		and
		\begin{eqnarray*}
			E\bigg(\frac{[V_1]_{m_1,p,q}[V_2]_{m_2,p,q}}{\big(p^{n+V_2} \ominus \beta_1\,q^{n+V_2}\big)^{-m_2}_{p,q}}\bigg)=\frac{[n+m_1+m_2-1]_{m_1+m_2,p,q}\,\beta^{m_1}_1\,\beta^{m_2}_2}{\big(p^{n} \ominus \beta_2\,q^{n}\big)^{m_1+m_2}_{p,q}},
		\end{eqnarray*}
		where $m_1\in\mathbb{N}\cup\{0\}$ and $m_2\in\mathbb{N}\cup\{0\}.$
		Furthermore,   the $(p,q)$-covariance of the functions  and $\widehat{V}:=(p^{n+V_2} \ominus \beta_2\,q^{n+V_2})[V_1]_{m_2,p,q}$ and $\overline{V}:=[V_2]_{m_1,p,q}$ is given by:
		\begin{eqnarray*}
			Cov\big(\widehat{V},\overline{V}\big)=\frac{[n]_{p,q}\beta_1\beta_2}{(p^{n} - \beta_2\,q^{n})}\bigg(\frac{[n+1]_{p,q}}{(p^{n+1} - \beta_2\,q^{n+1})}-\frac{[n]_{p,q}}{(p^{n}- \beta_2\,q^{n})}\bigg).
		\end{eqnarray*}
	\end{enumerate}
\subsection{Trinomial distribution and Chakrabarty and Jagannathan algebra \cite{CJ}}
 Putting $\mathcal{R}(x,y)=\frac{1-xy}{(p^{-1}-q)x},$ we obtain:
\begin{enumerate}
	\item[(i)]  The $(p^{-1},q)$-trinomial  distribution of the first kind
	is given by
	\begin{eqnarray*}
	P\big(Y_1=y_1,Y_2=y_2\big)=\genfrac{[}{]}{0pt}{}{n}{y_1,y_2}_{p^{-1},q}\frac{\alpha^{y_1}_1\alpha^{y_2}_2\,p^{-{n-y_1\choose 2}-{n-y_2\choose 2}}\,q^{{y_1\choose 2}+{y_2\choose 2}}}{(1 \oplus \alpha_1)^{n}_{p^{-1},q}(1 \oplus \alpha_2)^{n-y_1}_{p^{-1},q}}
	\end{eqnarray*}
	and the $(p^{-1},q)$-factorial moments are:
	\begin{small}
		\begin{eqnarray*}
			E\big([Y_1]_{m_1,p^{-1},q}\big)=\frac{[n]_{m_1,p^{-1},q}\,\alpha^{m_1}_1\,q^{m_1\choose 2}}{(1 \oplus \alpha_1)^{m_1}_{p^{-1},q}},\,m_1\in\{0,1,\cdots,n\},\\
			E\big([Y_2]_{m_2,p^{-1},q}|Y_1=y_1\big)=\frac{[n-y_1]_{m_2,p^{-1},q}\,\alpha^{m_2}_2\,q^{m_2\choose 2}}{(1 \oplus \alpha_2)^{m_2}_{p^{-1},q}},\,m_2\in\{0,1,\cdots,n-y_1\},\\
			E\big([Y_2]_{m_2,p^{-1},q}\big)=\frac{[n]_{m_2,p^{-1},q}\,\alpha^{m_2}_2\,q^{m_2\choose 2}\,p^{-{m_2\choose 2}-m_2(n-m_2)}}{(1 \oplus \alpha_2)^{m_2}_{p^{-1},q}(p^{n-m_2} \oplus \alpha_1\,q^{n-m_2})^{m_2}_{p^{-1},q}},\,m_2\in\{0,1,\cdots,n\},
		\end{eqnarray*}
		and 
		\begin{eqnarray*}
			E\big([Y_1]_{m_1,p^{-1},q}[Y_2]_{m_2,p^{-1},q}\big)
			=\frac{[n]_{m_1+m_2,p^{-1},q}\alpha^{m_1}_1\alpha^{m_2}_2\,p^{-{m_2\choose 2}-m_2(n-m_2)}q^{{m_1\choose 2}+{m_2 \choose 2}}}{(1 \oplus \alpha_1)^{m_1}_{p^{-1},q}(1 \oplus \alpha_2)^{m_2}_{p^{-1},q}(p^{n-m_2} \oplus \alpha_1\,q^{n-m_2})^{m_2}_{p^{-1},q}},
		\end{eqnarray*}
		where $m_1\in\{0,1,\cdots, n-m_2\}$ and $m_2\in\{0,1,\cdots, n\}.$ Moreover, the  $(p^{-1},q)$- covariance is derived as follows:
		\begin{eqnarray*}
			Cov\big([Y_1]_{p^{-1},q},[Y_2]_{p^{-1},q}\big)
			&=&\frac{p^{-n+1}\alpha_1\,\alpha_2\,[n]_{p^{-1},q}\big([n-1]_{p^{-1},q}-[n]_{p^{-1},q}\big)}{(1 + \alpha_1)(1 + \alpha_2)(p^{-n+1} + \alpha_1\,q^{n-1})}.
		\end{eqnarray*}
	\end{small}
	\item[(ii)] The negative  $(p^{-1},q)$-trinomial  distribution of the first kind, with parameters $n,$ $\underline{\alpha}=\big(\alpha_1,\alpha_2\big),$ $p,$ and, $q$ is presented by:
	\begin{small}
		\begin{eqnarray*}
			P\big(\underline{W}=\underline{w}\big)=\genfrac{[}{]}{0pt}{}{n+w_1+w_2-1}{w_1,w_2}_{p^{-1},q}\frac{\alpha^{w_1}_1\alpha^{w_2}_2\,p^{-{n-w_1\choose 2}-{n-w_2\choose 2}}\,q^{{w_1\choose 2}+{w_2\choose 2}}}{(1 \oplus \alpha_1)^{n+w_1+w_2}_{p^{-1},q}(1 \oplus \alpha_2)^{n+w_2}_{p^{-1},q}},
		\end{eqnarray*}
	\end{small}
	where $w_j\in\mathbb{N}\cup\{0\}, 0<\alpha_j<1$, and $j\in\{1,2\}.$
	Besides, for $m_1\in\mathbb{N}\cup\{0\}$ and $m_2\in\mathbb{N}\cup\{0\},$ 
	
	its $(p,q^{-1})$-factorial moments  are given as follows:
	\begin{small}
		\begin{eqnarray*}
			E\big([W_2]_{m_2,p,q^{-1}}\big)=[n+m_2-1]_{m_2,p^{-1},q}\,\alpha^{m_2}_2,\,m_2\in\mathbb{N}\cup\{0\},
		\end{eqnarray*}
		\begin{eqnarray*}
			E\big([W_1]_{m_1,p,q^{-1}}|W_2=w_2\big)=[n+w_2+m_1-1]_{m_1,p^{-1},q}\,\alpha^{m_1}_1,\,m_1\in\mathbb{N}\cup\{0\},
		\end{eqnarray*}
		\begin{eqnarray*}
			E\bigg(\frac{[W_1]_{m_1,p,q^{-1}}}{\big(p^{n+W_2} \oplus \alpha_2\,q^{n+W_2}\big)^{m_1}_{p^{-1},q}}\bigg)=[n+m_1-1]_{m_1,p^{-1},q}\,\alpha^{m_1}_1,\,m_1\in\mathbb{N}\cup\{0\},
		\end{eqnarray*}
		and 
		\begin{eqnarray*}
			E\bigg(\frac{[W_1]_{m_1,p,q^{-1}}\,[W_2]_{m_2,p,q^{-1}}}{\big(p^{n+W_2} \oplus \alpha_2\,q^{n+W_2}\big)^{m_1}_{p^{-1},q}}\bigg)
			=[n+m_1+m_2-1]_{m_1+m_2,p^{-1},q}\,\alpha^{m_1}_1\alpha^{m_2}_2,
		\end{eqnarray*}
	\end{small}
	Furthermore, 
	the covariance of $\widehat{W}:=(p^{-n-W_2} - \alpha_2\,q^{n+W_2})^{-1}[W_1]_{m_1,p,q^{-1}}$ and $\overline{W}:=[W_2]_{m_2,p,q^{-1}}$ is determined by:
	\begin{eqnarray*}
		Cov\big(\widehat{W},\overline{W}\big)=[n]_{p^{-1},q}\alpha_1\alpha_2\big([n+1]_{p^{-1},q}-[n]_{p^{-1},q}\big).
	\end{eqnarray*}
	\item[(iii)]  The $(p^{-1},q)$-trinomial distribution of the second kind, with parameters $n,$ $\underline{\beta}=\big(\beta_1,\beta_2\big),$ $p,$ and, $q$ is given by:
	\begin{small}
		\begin{eqnarray*}
			P\big(X_1=x_1,X_2=x_2\big)=\genfrac{[}{]}{0pt}{}{n}{x_1,x_2}_{p^{-1},q}\,\beta^{x_1}_1\,\beta^{x_2}_2(1 \ominus \beta_1)^{n-x_1}_{p^{-1},q}(1 \ominus \beta_2)^{n-x_1-x_2}_{p^{-1},q},
		\end{eqnarray*}
	\end{small}
	where $x_j\in\{0,1,\cdots,n\}, x_1+x_2\leq n, s_j=\displaystyle\sum_{i=1}^{j}x_j, 0<\beta_j<1$, and $j\in\{1,2\}.$
	Besides, its 
	$(p^{-1},q)$-factorial moments  are given by:
	\begin{eqnarray*}
		E\big([X_1]_{m_1,p^{-1},q}\big)=[n]_{m_1,p^{-1},q}\,\beta^{m_1}_1,\,m_1\in\{0,1,\cdots,n\},
	\end{eqnarray*}
	\begin{eqnarray*}
		E\big([X_2]_{m_2,p^{-1},q}|X_1=x_1\big)=[n-x_1]_{m_2,p^{-1},q}\,\beta^{m_2}_2,\,m_2\in\{0,1,\cdots,n-x_1\},
	\end{eqnarray*}
	\begin{eqnarray*}
		E\bigg(\frac{[X_2]_{m_2,p^{-1},q}}{\big(p^{n-m_2-X_1} \ominus \beta_1\,q^{n-m_2-X_1}\big)^{m_2}_{p^{-1},q}}\bigg)=[n]_{m_2,p^{-1},q}\,\beta^{m_2}_2, \, m_2\in\{0,1,\cdots,n\},
	\end{eqnarray*}
	and
	\begin{eqnarray*}
		E\bigg(\frac{[X_1]_{m_1,p^{-1},q}[X_2]_{m_2,p^{-1},q}}{\big(p^{n-m_2-X_1} \ominus \beta_1\,q^{n-m_2-X_1}\big)^{m_2}_{p^{-1},q}}\bigg)=[n]_{m_1+m_2,p^{-1},q}\,\beta^{m_1}_1\,\beta^{m_2}_2,
	\end{eqnarray*}
	where $m_1\in\{0,1,\cdots,n-m_2\}$ and $m_2\in\{0,1,\cdots,n\}.$ Moreover, 
	the  $(p^{-1},q)$- covariance of the functions  $\big(p^{-n+1+X_1} - \beta_1\,q^{n-1-X_1}\big)^{-1}[X_2]_{p^{-1},q}$ and $[X_1]_{p^{-1},q}$  is given by:
	\begin{small}
		\begin{eqnarray*}
			Cov\bigg[[X_1]_{p^{-1},q}, \big(p^{-n+1+X_1} - \beta_1\,q^{n-1-X_1}\big)^{-1}[X_2]_{p^{-1},q}\bigg]=[n]_{p^{-1},q}\beta_1\beta_2\big([n-1]_{p^{-1},q}-[n]_{p^{-1},q}\big).
		\end{eqnarray*}
	\end{small} 
	\item[(iii)]
	The negative $(p^{-1},q)$- trinomial probability distribution of the second kind, with parameters $n,$ $\underline{\beta}=\big(\beta_1,\beta_2\big),$ $p,$ and, $q$  is given by:
	\begin{small}
		\begin{eqnarray*}
			P\big(\underline{V}=\underline{v}\big)=\genfrac{[}{]}{0pt}{}{n+v_1+v_2-1}{v_1,v_2}_{p^{-1},q}\beta^{v_1}_1\beta^{v_2}_2(1 \ominus \beta_1)^{n+v_2}_{p^{-1},q}(1 \ominus \beta_2)^{n}_{p^{-1},q},\,v_j\in\mathbb{N}\cup\{0\},
		\end{eqnarray*}
	\end{small}
	where $ 0<\beta_j<1$, and $j\in\{1,2\}.$
	Moreover, 
	its $(p^{-1},q)$-factorial moments of  are presented as follows:
	\begin{eqnarray*}
		E\big([V_2]_{m_2,p^{-1},q}\big)=\frac{[n+m_2-1]_{m_2,p^{-1},q}\,\beta^{m_2}_2}{\big(p^{n} \ominus \beta_2\,q^{n}\big)^{m_2}_{p^{-1},q}},
	\end{eqnarray*}
	\begin{eqnarray*}
		E\big([V_1]_{m_,p^{-1},q}|V_2=v_2\big)=\frac{[n+v_2+m_1-1]_{m_1,p^{-1},q}\,\beta^{m_1}_1}{\big(p^{n+v_2} \ominus \beta_1\,q^{n+v_2}\big)^{m_1}_{p^{-1},q}},
	\end{eqnarray*}
	\begin{eqnarray*}
		E\bigg(\frac{[V_1]_{m_1,p^{-1},q}}{\big(p^{n+V_2} \ominus \beta_1\,q^{n+V_2}\big)^{-m_1}_{p^{-1},q}}\bigg)=\frac{[n+m_1-1]_{m_1,p^{-1},q}\,\beta^{m_1}_1}{\big(p^{n} \ominus \beta_2\,q^{n}\big)^{m_1}_{p^{-1},q}},
	\end{eqnarray*}
	and
	\begin{eqnarray*}
		E\bigg(\frac{[V_1]_{m_1,p^{-1},q}[V_2]_{m_2,p^{-1},q}}{\big(p^{n+V_2} \ominus \beta_1\,q^{n+V_2}\big)^{-m_2}_{p^{-1},q}}\bigg)=\frac{[n+m_1+m_2-1]_{m_1+m_2,p^{-1},q}\,\beta^{m_1}_1\,\beta^{m_2}_2}{\big(p^{n} \ominus \beta_2\,q^{n}\big)^{m_1+m_2}_{p^{-1},q}},
	\end{eqnarray*}
	where $m_1\in\mathbb{N}\cup\{0\}$ and $m_2\in\mathbb{N}\cup\{0\}.$
	Furthermore,   the $(p^{-1},q)$-covariance of the functions  and $\widehat{V}:=(p^{-n-V_2} - \beta_2\,q^{n+V_2})[V_1]_{m_2,p^{-1},q}$ and $\overline{V}:=[V_2]_{m_1,p^{-1},q}$ is given by:
	\begin{eqnarray*}
		Cov\big(\widehat{V},\overline{V}\big)=\frac{[n]_{p^{-1},q}\beta_1\beta_2}{(p^{-n} - \beta_2\,q^{n})}\bigg(\frac{[n+1]_{p^{-1},q}}{(p^{-n-1} - \beta_2\,q^{n+1})}-\frac{[n]_{p^{-1},q}}{(p^{-n}- \beta_2\,q^{n})}\bigg).
	\end{eqnarray*}
\end{enumerate}
	\subsection{Trinomial distribution and Hounkonnou-Ngompe generalization of $q-$ Quesne algebra \cite{Hounkonnou&Ngompe07a} }
	Setting $\mathcal{R}(x,y)=\frac{xy-1}{(q-p^{-1})y},$ we obtain the following results:
	\begin{enumerate}
		\item[(i)] The Hounkonnou-Ngompe generalization of $q$- Quesne-trinomial  distribution of the first kind
		is given by:
		\begin{eqnarray*}
		P\big(Y_1=y_1,Y_2=y_2\big)=\genfrac{[}{]}{0pt}{}{n}{y_1,y_2}^Q_{p,q}\frac{\alpha^{y_1}_1\alpha^{y_2}_2\,p^{{n-y_1\choose 2}+{n-y_2\choose 2}}\,q^{-{y_1\choose 2}-{y_2\choose 2}}}{(1 \oplus \alpha_1)^{n}_{p,q}(1 \oplus \alpha_2)^{n-y_1}_{p,q}}
		\end{eqnarray*}
		and their factorial moments are:
		\begin{small}
			\begin{eqnarray*}
				E\big([Y_1]^Q_{m_1,p,q}\big)=\frac{[n]^Q_{m_1,p,q}\,\alpha^{m_1}_1\,q^{-{m_1\choose 2}}}{(1 \oplus \alpha_1)^{m_1}_{p,q}},\,m_1\in\{0,1,\cdots,n\},\\
				E\big([Y_2]^Q_{m_2,p,q}|Y_1=y_1\big)=\frac{[n-y_1]^Q_{m_2,p,q}\,\alpha^{m_2}_2\,q^{-{m_2\choose 2}}}{(1 \oplus \alpha_2)^{m_2}_{p,q}},\,m_2\in\{0,1,\cdots,n-y_1\},\\
				E\big([Y_2]^Q_{m_2,p,q}\big)=\frac{[n]^Q_{m_2,p,q}\,\alpha^{m_2}_2\,p^{{m_2\choose 2}+m_2(n-m_2)}\,q^{-{m_2\choose 2}}}{(1 \oplus \alpha_2)^{m_2}_{p,q}(p^{n-m_2} \oplus \alpha_1\,q^{n-m_2})^{m_2}_{p,q}},\,m_2\in\{0,1,\cdots,n\},
			\end{eqnarray*}
			and 
			\begin{eqnarray*}
				E\big([Y_1]^Q_{m_1,p,q}[Y_2]^Q_{m_2,p,q}\big)
				=\frac{[n]^Q_{m_1+m_2,p,q}\alpha^{m_1}_1\alpha^{m_2}_2\,p^{{m_2\choose 2}+m_2(n-m_2)}\,q^{-{m_1\choose 2}-{m_2 \choose 2}}}{(1 \oplus \alpha_1)^{m_1}_{p,q}(1 \oplus \alpha_2)^{m_2}_{p,q}(p^{n-m_2} \oplus \alpha_1\,q^{n-m_2})^{m_2}_{p,q}},
			\end{eqnarray*}
			where $m_1\in\{0,1,\cdots, n-m_2\}$ and $m_2\in\{0,1,\cdots, n\}.$ Moreover, the   covariance is derived as follows:
			\begin{eqnarray*}
				Cov\big([Y_1]^Q_{p,q},[Y_2]^Q_{\mathcal{R}(p,q)}\big)
				=\frac{p^{n-1}[n]^Q_{p,q}\,\alpha_1\,\alpha_2\big([n-1]^Q_{p,q}-[n]^Q_{p,q}\big)}{(1 + \alpha_1)(1 + \alpha_2)(p^{n-1} + \alpha_1\,q^{1-n})}.
			\end{eqnarray*}
		\end{small}
		\item[(ii)] The negative  Hounkonnou-Ngompe generalization of $q-$ Quesne-trinomial  distribution of the first kind, with parameters $n,$ $\underline{\alpha}=\big(\alpha_1,\alpha_2\big),$ $p,$ and, $q$ is presented by:
		\begin{small}
			\begin{eqnarray*}
				P\big(\underline{W}=\underline{w}\big)=\genfrac{[}{]}{0pt}{}{n+w_1+w_2-1}{w_1,w_2}^Q_{p,q}\frac{\alpha^{w_1}_1\alpha^{w_2}_2\,p^{{n-w_1\choose 2}+{n-w_2\choose 2}}\,q^{-{w_1\choose 2}-{w_2\choose 2}}}{(1 \oplus \alpha_1)^{n+w_1+w_2}_{p,q}(1 \oplus \alpha_2)^{n+w_2}_{p,q}},
			\end{eqnarray*}
		\end{small}
		where $w_j\in\mathbb{N}\cup\{0\}, 0<\alpha_j<1$, and $j\in\{1,2\}.$
		Besides,
		its factorial moments  are given as follows:
		\begin{small}
			\begin{eqnarray*}
				E\big([W_2]^Q_{m_2,p^{-1},q^{-1}}\big)=\alpha^{m_2}_2\,[n+m_2-1]^Q_{m_2,p,q},
			\end{eqnarray*}
			\begin{eqnarray*}
				E\big([W_1]^Q_{m_1,p^{-1},q^{-1}}|W_2=w_2\big)=\alpha^{m_1}_1\,[n+w_2+m_1-1]^Q_{m_1,p,q},
			\end{eqnarray*}
			\begin{eqnarray*}
				E\bigg(\frac{[W_1]^Q_{m_1,p^{-1},q^{-1}}}{\big(p^{n+W_2} \oplus \alpha_2\,q^{n+W_2}\big)^{m_1}_{p,q}}\bigg)=\alpha^{m_1}_1\,[n+m_1-1]_{m_1,p,q},
			\end{eqnarray*}
			and 
			\begin{small}
				\begin{eqnarray*}
					E\bigg(\frac{[W_1]^Q_{m_1,p^{-1},q^{-1}}\,[W_2]^Q_{m_2,p^{-1},q^{-1}}}{\big(p^{n+W_2} \oplus \alpha_2\,q^{n+W_2}\big)^{m_1}_{p,q}}\bigg)
					=\alpha^{m_1}_1\alpha^{m_2}_2\,[n+m_1+m_2-1]^Q_{m_1+m_2,p,q},
				\end{eqnarray*}
			\end{small}
			where $m_1\in\mathbb{N}\cup\{0\}$ and $m_2\in\mathbb{N}\cup\{0\}.$
		\end{small}
		Furthermore, 
		the covariance of $\widehat{W}:=(p^{n+W_2} \oplus \alpha_2\,q^{n+W_2})^{-1}[W_1]^Q_{m_1,p^{-1},q^{-1}}$ and $\overline{W}:=[W_2]^Q_{m_2,p^{-1},q^{-1}}$ is determined by:
		\begin{eqnarray*}
			Cov\big(\widehat{W},\overline{W}\big)=[n]^Q_{p,q}\alpha_1\alpha_2\big([n+1]^Q_{p,q}-[n]^Q_{p,q}\big).
		\end{eqnarray*}
		\item[(iii)]  The Hounkonnou-Ngompe generalization of $q-$ Quesne-trinomial distribution of the second kind, with parameters $n,$ $\underline{\beta}=\big(\beta_1,\beta_2\big),$ $p,$ and, $q$ is given by:
		\begin{small}
			\begin{eqnarray*}
				P\big(X_1=x_1,X_2=x_2\big)=\genfrac{[}{]}{0pt}{}{n}{x_1,x_2}^Q_{p,q}\,\beta^{x_1}_1\,\beta^{x_2}_2(1 \ominus \beta_1)^{n-x_1}_{p,q}(1 \ominus \beta_2)^{n-x_1-x_2}_{p,q},
			\end{eqnarray*}
		\end{small}
		where $x_j\in\{0,1,\cdots,n\}, x_1+x_2\leq n, s_j=\displaystyle\sum_{i=1}^{j}x_j, 0<\beta_j<1$, and $j\in\{1,2\}.$
		Besides, its 
		factorial moments  are given by:
		\begin{eqnarray*}
			E\big([X_1]^Q_{m_1,p,q}\big)=[n]^Q_{m_1,p,q}\,\beta^{m_1}_1,\,m_1\in\{0,1,\cdots,n\},
		\end{eqnarray*}
		\begin{eqnarray*}
			E\big([X_2]^Q_{m_2,p,q}|X_1=x_1\big)=[n-x_1]^Q_{m_2,p,q}\,\beta^{m_2}_2,\,m_2\in\{0,1,\cdots,n-x_1\},
		\end{eqnarray*}
		\begin{eqnarray*}
			E\bigg(\frac{[X_2]^Q_{m_2,p,q}}{\big(p^{n-m_2-X_1} \ominus \beta_1\,q^{n-m_2-X_1}\big)^{m_2}_{p,q}}\bigg)=[n]^Q_{m_2,p,q}\,\beta^{m_2}_2, \, m_2\in\{0,1,\cdots,n\},
		\end{eqnarray*}
		and
		\begin{eqnarray*}
			E\bigg(\frac{[X_1]^Q_{m_1,p,q}[X_2]^Q_{m_2,p,q}}{\big(p^{n-m_2-X_1} \ominus \beta_1\,q^{n-m_2-X_1}\big)^{m_2}_{p,q}}\bigg)=[n]^Q_{m_1+m_2,p,q}\,\beta^{m_1}_1\,\beta^{m_2}_2,
		\end{eqnarray*}
		where $m_1\in\{0,1,\cdots,n-m_2\}$ and $m_2\in\{0,1,\cdots,n\}.$ Moreover, 
		the   covariance of the functions  $\big(p^{n-1-X_1} - \beta_1\,q^{1+X_1-n}\big)^{-1}[X_2]^Q_{p,q}$ and $[X_1]^Q_{p,q}$  is given by:
		\begin{small}
			\begin{eqnarray*}
				Cov\bigg[[X_1]^Q_{p,q}, \big(p^{n-1-X_1} - \beta_1\,q^{1+X_{1}-n}\big)^{-1}[X_2]^Q_{p,q}\bigg]=[n]^Q_{p,q}\beta_1\beta_2\big([n-1]^Q_{p,q}-[n]^Q_{p,q}\big).
			\end{eqnarray*}
		\end{small} 
		\item[(iii)]
		The negative Hounkonnou-Ngompe generalization of $q-$ Quesne- trinomial probability distribution of the second kind, with parameters $n,$ $\underline{\beta}=\big(\beta_1,\beta_2\big),$ $p,$ and, $q$  is given by:
		\begin{small}
			\begin{eqnarray*}
				P\big(\underline{V}=\underline{v}\big)=\genfrac{[}{]}{0pt}{}{n+v_1+v_2-1}{v_1,v_2}^Q_{p,q}\beta^{v_1}_1\beta^{v_2}_2(1 \ominus \beta_1)^{n+v_2}_{p,q}(1 \ominus \beta_2)^{n}_{p,q},\,v_j\in\mathbb{N}\cup\{0\},
			\end{eqnarray*}
		\end{small}
		where $ 0<\beta_j<1$, and $j\in\{1,2\}.$
		Moreover, for $m_1\in\mathbb{N}\cup\{0\}$ and $m_2\in\mathbb{N}\cup\{0\},$ 
		its factorial moments of  are presented as follows:
		\begin{eqnarray*}
			E\big([V_2]^Q_{m_2,p,q}\big)=\frac{[n+m_2-1]^Q_{m_2,p,q}\,\beta^{m_2}_2}{\big(p^{n} \ominus \beta_2\,q^{n}\big)^{m_2}_{p,q}},
		\end{eqnarray*}
		\begin{eqnarray*}
			E\big([V_1]^Q_{m_,p,q}|V_2=v_2\big)=\frac{[n+v_2+m_1-1]^Q_{m_1,p,q}\,\beta^{m_1}_1}{\big(p^{n+v_2} \ominus \beta_1\,q^{n+v_2}\big)^{m_1}_{p,q}},
		\end{eqnarray*}
		\begin{eqnarray*}
			E\bigg(\frac{[V_1]^Q_{m_1,p,q}}{\big(p^{n+V_2} \ominus \beta_1\,q^{n+V_2}\big)^{-m_1}_{p,q}}\bigg)=\frac{[n+m_1-1]^Q_{m_1,p,q}\,\beta^{m_1}_1}{\big(p^{n} \ominus \beta_2\,q^{n}\big)^{m_1}_{p,q}},
		\end{eqnarray*}
		and
		\begin{eqnarray*}
			E\bigg(\frac{[V_1]^Q_{m_1,p,q}[V_2]^Q_{m_2,p,q}}{\big(p^{n+V_2} \ominus \beta_1\,q^{n+V_2}\big)^{-m_2}_{p,q}}\bigg)=\frac{[n+m_1+m_2-1]^Q_{m_1+m_2,p,q}\,\beta^{m_1}_1\,\beta^{m_2}_2}{\big(p^{n} \ominus \beta_2\,q^{n}\big)^{m_1+m_2}_{p,q}}.
		\end{eqnarray*}
		Furthermore,   the covariance of the functions  $\widehat{V}:=(p^{n+V_2} \ominus \beta_2\,q^{n+V_2})[V_1]^Q_{m_2,p,q}$ and $\overline{V}:=[V_2]^Q_{m_1,p,q}$ is given by:
		\begin{eqnarray*}
			Cov\big(\widehat{V},\overline{V}\big)=\frac{[n]^Q_{p,q}\beta_1\beta_2}{(p^{n} - \beta_2\,q^{n})}\bigg(\frac{[n+1]^Q_{p,q}}{(p^{n+1} - \beta_2\,q^{n+1})}-\frac{[n]^Q_{p,q}}{(p^{n}- \beta_2\,q^{n})}\bigg).
		\end{eqnarray*}
	\end{enumerate}
\section*{Acknowledgments}
This research was partly supported by the SNF Grant No. IZSEZ0\_206010.


\begin{thebibliography}{99}
	
\bibitem{BC}L. C. Biedenharn, The quantum group $SU_{q}(2)$ and a $q-$analogue of the boson
operators, {\it J. Phys. A} {\bf 22} (1989), L873-L878.
\bibitem{CJ} R. Chakrabarti and R. Jagannathan, {A $(p, q)-$oscillator realisation of two-parameter quantum algebras},  {\it  J. Phys.  A: Math. Gen.} {\bf 24}, L711-L718,  IMSC-91-15 (1991).
\bibitem{CA1}Ch. A. Charalambides, {\it Discrete $q-$ distributions.}  John Wiley and Sons, Inc., Hoboken, New Jersey, 2016.
\bibitem{CA2} Charalambides, Ch. A.  $q$-Multinomial and negative $q$-multinomial distributions, {\em Comm. Statist. Theory Meth.} (2020).	
\bibitem{CA3} Charalambides, Ch. A.   Multivariate $q-$ P\'olya and inverse $q-$ P\'olya  distributions, {\em Comm. Statist. Theory Meth.} (2020).
\bibitem{CA4} Charalambides, Ch. A.   $q-$ Factorial moments of bivariate discrete distributions, {\em Comm. Statist. Theory Meth.} (2022).
\bibitem{HMD} M. N. Hounkonnou  and F. Melong, $\mathcal{R}(p,q)$ Analogs of Discrete Distributions: 
General Formalism and Applications, {\it Journal of Stochastic Analysis:} Vol. 1 : No. 4 , Article 11. 
DOI: 10.31390/josa.1.4.11, (2020).
\bibitem{HMRC} M. N. Hounkonnou  and F. Melong, $\mathcal{R}(p,q)$-deformed  combinatorics: full characterization and illustration, arXiv:1906.03059. 
\bibitem{HB} M. N. Hounkonnou  and J. D. Kyemba Bukweli, $\mathcal{R}(p,q)$-calculus: differentiation and integration, {\it SUT. J. Math.} {\bf 49}, 145-167 (2013). 
\bibitem{HB1}M. N. Hounkonnou  and J. D. Kyemba Bukweli, $(R,p,q)$-deformed quantum algebras: Coherent states and special functions, {\it J. Math. Phys.} \textbf{51} , 063518-063518.20 (2010).

\bibitem{Hounkonnou&Ngompe07a} M. N. Hounkonnou and E. B. Ngompe Nkouankam, {New $(p, q, \mu, \nu, f)$-deformed states,}  {\it  J. Phys. A: Math. Theor.}  {\bf 40}, 12113-12130 (2007).
\bibitem{JS}R. Jagannathan and K. Srinivasa Rao, Two-parameter quantum algebras,
twin-basic numbers, and associated generalized hypergeometric series, {\it Proceedings of the International
	Conference on Number Theory and Mathematical Physics}, 20-21 December 2005.
\bibitem{M}A. J. Macfarlane, On $q-$ analogues of the quantum harmonic oscillator and
quantum group $SU_q(2),$ 
J. Phys. A {\bf 22} (1989), 4581-4588.
\bibitem{TN}T, Nishino,  {\it Function theory in several complex variables}, Translations of mathematical monographs, Volume 193, American Mathematical Society, Providence, Rhode Island, 2001.
%
%
%
%
%
%
%
%
%
%
%
%
%
%
%
%
%
%
%
%
%
%
%
%
%

\end{thebibliography}
\end{document}